\documentclass[10pt]{article}
\usepackage[T1]{fontenc}
\usepackage{amsmath,amsfonts,amsthm,mathrsfs,amssymb}
\usepackage{cite}
\usepackage{graphicx,float}
\usepackage{subfigure}
\usepackage{placeins}
\usepackage{color}
\usepackage{indentfirst}
\numberwithin{equation}{section}
\newcommand{\R}{{\mathbb R}}
\newcommand{\Z}{{\mathbb Z}}
\newcommand{\N}{{\mathbb N}}
\newcommand{\C}{{\mathbb C}}
\newcommand{\s}{{\mathbb S}}

\newcommand{\bb}{\bf}

\newcommand{\no}{\nonumber}
\newcommand{\be}{\begin{eqnarray}}
\newcommand{\ben}{\begin{eqnarray*}}
\newcommand{\en}{\end{eqnarray}}
\newcommand{\enn}{\end{eqnarray*}}

\newcommand{\pa}{\partial}

\newcommand{\ov}{\overline}
\newcommand{\curl}{{\rm curl\,}}

\newcommand{\grad}{{\rm grad\,}}
\newcommand{\divv}{{\rm div\,}}
\newcommand{\real}{{\rm Re\,}}
\newcommand{\Ima}{{\rm Im\,}}

\newcommand{\G}{\Gamma}

\newcommand{\Om}{\Omega}
\newcommand{\om}{\omega}

\newtheorem{theorem}{Theorem}[section]
\newtheorem{lemma}[theorem]{Lemma}
\newtheorem{corollary}[theorem]{Corollary}
\newtheorem{definition}[theorem]{Definition}
\newtheorem{remark}[theorem]{Remark}

\definecolor{rot}{rgb}{0.000,0.000,0.000}

\begin{document}
\renewcommand{\theequation}{\arabic{section}.\arabic{equation}}
\begin{titlepage}
\title{\bf Direct and Inverse Elastic Scattering From Anisotropic Media}


\author{
Gang Bao\ \ ({\sf baog@zju.edu.cn}) \ \\
 {\small School of Mathematical Sciences, Zhejiang University, Hangzhou 310027, China}\\ \\
Guanghui Hu\ \ ({\sf hu@csrc.ac.cn}) \ \\
 {\small Beijing Computational Science Research Center, Beijing 100094, China}\\ \\
Jiguang Sun \ \ ({\sf jiguangs@mtu.edu}) \ \\
{\small Department of Mathematical Sciences, Michigan Technological University, Houghton, MI 49931 U.S.A.} \\ \\
Tao Yin\ \ ({\sf taoyin@zju.edu.cn}) \ \\
 {\small School of Mathematical Sciences, Zhejiang University, Hangzhou 310027, China}
}
\date{}
\end{titlepage}
\maketitle
\vspace{.2in}

\begin{abstract}
Assume a time-harmonic elastic wave is incident onto a penetrable anisotropic body embedded into a homogeneous isotropic background medium. We propose an equivalent variational formulation in a truncated bounded domain and show \textcolor{rot}{the} uniqueness and existence of weak solutions by applying the Fredholm alternative and using properties of the Dirichlet-to-Neumann \textcolor{rot}{map} in both two and three dimensions. The Fr\'echet derivative of the near-field solution operator with respect to the \textcolor{rot}{boundary of the scatterer} is derived. As an application, we design a descent algorithm for recovering the interface from the near-field data of one or several incident directions and frequencies. Numerical examples in 2D are \textcolor{rot}{presented} to show the validity and accuracy of our methods.

\vspace{.2in} {\bf Keywords:} Linear elasticity, Lam\'e system, variational approach, Fr\'echet derivative, Dirichlet-to-Neumann map, inverse scattering.
\end{abstract}

\section{Introduction}
Time-harmonic elastic scattering problems arise from many \textcolor{rot}{mechanic} systems and engineering structures, in which the linear elasticity theory provides an essential tool for analysis and design. \textcolor{rot}{For} an infinite background medium,
the boundary value problem for the Lam\'e system  \textcolor{rot}{can} be reduced to an equivalent system \textcolor{rot}{on}
a bounded \textcolor{rot}{domain}. For instance, \textcolor{rot}{the finite element method for the} scattering problems
usually requires a strongly elliptic variational formulation with a nonlocal boundary condition (see e.g.,\cite{CXZ,BPT}). To truncate the unbounded domain, one needs to derive the so-called Dirichlet-to-Neumann map (or non-reflecting boundary condition, transparent boundary operator) on an artificial boundary as a replacement of the Kupradze radiation condition at infinity. In the literature, \textcolor{rot}{the} DtN map in elastodynamics have been used by some physicists and engineers for simulation (\cite{AGH,OLBC10,GG03,BC,GK1990,GK1995}). However,
properties of the transparent operator have not been sufficiently investigated yet.
These properties are fundamental \textcolor{rot}{for} the strong ellipticity of the sesquilinear form generated
by the variational formulation and \textcolor{rot}{the} well-posedness (existence, uniqueness and stability) of the scattering \textcolor{rot}{problem}. We refer to \cite{CK98,CM, Nedelec} for the treatment of the time-harmonic Helmholtz and Maxwell equations.
Although a nonlocal boundary
condition closely related to the DtN \textcolor{rot}{map} was utilized in \cite{BP2008}, mapping properties of the non-reflecting operator in Sobolev spaces were not involved there. In a recent paper \cite{Li2016}, a special sesquilinear form, which corresponds to the choice of the parameters $\tilde{\lambda}=\lambda+\mu$, $\tilde{\mu}=0$ in Betti's formula (\ref{GStress}),  has been employed to prove well-posedeness of the elastic scattering problem. However, the approach of \cite{Li2016} relies heavily on the boundary condition of the scatterer and applies only to a rigid impenetrable elastic body in two dimensions.

This paper is concerned with both direct and inverse scattering from an anisotropic elastic body in a homogenous isotropic background medium. The first half is devoted to \textcolor{rot}{the} well-posedness in a more general setting.
We propose an equivalent variational formulation \textcolor{rot}{on} a truncated bounded domain, and show \textcolor{rot}{the} uniqueness and existence of weak solutions \textcolor{rot}{for both} inhomogeneous penetrable anisotropic bodies \textcolor{rot}{and impenetrable} scatterers with \textcolor{rot}{various} boundary conditions. In contrast to the Helmholtz equation, the real part of the DtN \textcolor{rot}{map}, $\real \mathcal{T}$, is not negative-definite. Nevertheless, the resulting sesquilinear
form \textcolor{rot}{is still} strongly elliptic since the operator $-\real \mathcal{T}$ can be decomposed into
the sum of a positive-definite operator and a finite-dimensional operator; see Lemma \ref{Lem:generalizedDtN} (ii) and Lemma \ref{Lem:generalizedDtN-3D} (ii).
 Motivated by Betti's formula, we analyze the DtN \textcolor{rot}{map} for the \emph{generalized} stress operator
 which covers \textcolor{rot}{the} usual stress operator (i.e., $\tilde{\lambda}=\lambda$, $\tilde{\mu}=\mu$ in (\ref{GStress})) and the special case discussed in \cite{Li2016,EH}. To prove uniqueness, we verify
 the Rellich's identity in elasticity; see Lemma \ref{lem:Rellich-2D} and Lemma \ref{Lem:generalizedDtN-3D} (iii).
 Our proof \textcolor{rot}{is new in the sense that it generalized} the arguments \textcolor{rot}{in} \cite{Li2016} and \cite[Lemma 5.8]{Hahner} for special cases. The Rellich's identity in periodic structures can be found in \cite{EH2012,EH}.

 The second half of this paper is concerned with the inverse problem of \textcolor{rot}{reconstructing} the shape of an unknown anisotropic body. Relying on the variational arguments presented \textcolor{rot}{in} the first half and those in \cite{Hettlich} and \cite{Kirsch1993},
we derive the Fr\'echet derivative of the solution operator with respect to the scattering surface. A different approach based on the integral equation was used in acoustics \cite{Potthast1996} and in elasticity \cite{Louer2012,C1995}.
  \textcolor{rot}{The} shape derivative can be \textcolor{rot}{used to} design a \textcolor{rot}{nonlinear} optimization approach for shape recovery from the data of several incident directions and frequencies. We employ a decent method \textcolor{rot}{to find} the parameters of the unknown surface in a finite dimensional space. At each iteration step, \textcolor{rot}{the forward problem needs to be solved} and \textcolor{rot}{the correctness of the parameters needs to be evaluated}. Numerical examples in 2D are \textcolor{rot}{presented} to show the validity and accuracy of our inversion algorithms. We refer to the review article \cite{BC} and the recent monograph \cite{Ammari2015} for \textcolor{rot}{various} inverse problems in elasticity and to \cite{B, BLT,BY, Li2016} where \textcolor{rot}{iterative approaches} using multi-frequency data \textcolor{rot}{were} developed.

It is worth noting that there are still two open \textcolor{rot}{questions}. Firstly, how to derive a frequency-dependent estimate of the solution for a star-shaped rigid scatterer? Readers are referred to e.g. \cite{CM} for a wavenumber-dependent estimate in the acoustic case, which was derived based on the use of a Rellich-type identity for the scalar Helmholtz equation. In linear elasticity, the lack of the positivity of  $-\real \mathcal{T}$ leads to essential difficulties in generalizing the arguments of \cite{CM}. Secondly, how to prove \textcolor{rot}{the} well-posedness in a homogenous anisotropic background medium? A new radiation condition at infinity \textcolor{rot}{seems to be necessary}, which should cover the Kupradze radiation condition as a special case. In this paper, the assumption of the isotropic background medium has considerably simplified our arguments. The far-field asymptotics of the Green's tensor
for a transversely isotropic solid was recently analyzed in \cite{Gridin}. However, a radiation condition in the general case seems \textcolor{rot}{unavailable} in the literature.

The remaining part of this paper is organized as following. In Section \ref{sec:2}, we describe the forward scattering model in $\R^N$ ($N=2,3$) and prove the unique solvability using variational arguments. Properties of the DtN \textcolor{rot}{map} in two and three dimensions will be presented in Sections \ref{prof-dtn-2d} and \ref{prof-dtn-3d}, respectively.
In Section \ref{sec:inverse} we derive the Fr\'echet derivative and apply it to \textcolor{rot}{solve the} inverse scattering problems. Numerical tests for both direct and inverse problems will be reported in \textcolor{rot}{Section} \ref{sec:Numerics}.

\section{Well-posedness of \textcolor{rot}{the} direct scattering problems}\label{sec:2}
\subsection{Mathematical formulations}
Suppose that a time-harmonic elastic wave $u^{in}$ (with the time variation of the form $e^{-i\omega t}$ where $\omega>0$ is a fixed frequency) is incident onto an anisotropic elastic body $\Om$
embedded in an infinite homogeneous isotropic background medium in $\R^N$ ($N=2,3$).
It is assumed that $\Om$
is a bounded Lipschitz domain and the exterior $\Omega^c:=\R^N\backslash\ov{\Om}$ of $\Om$ is connected.
In particular, $\Omega$ is allowed to consist of a finite number of disconnected bounded components.
In linear elasticity, the spatially-dependent displacement vector $u(x)=(u_1,u_2,\cdots, u_N)\textcolor{rot}{^\top}(x)$\textcolor{rot}{, where $(\cdot)^\top$ means the transpose,} is governed by the following reduced Lam\'e system
\be\label{I}
\displaystyle\sum_{j,k,l=1}^N \frac{\partial}{\partial x_{\textcolor{rot}{j}}} \left( C_{ijkl}(x)  \frac{\partial u_k(x)}{\partial x_l}  \right)    +\omega^2\rho(x)\, u_i(x)=0\quad \mbox{in}\quad \R^N,\quad i=1,2, \cdots, N.
\en
In (\ref{I}), $u=u^{in}+u^{sc}$ is the total field and $u^{sc}$ is the \textcolor{rot}{scattered} field; $\mathcal{C}=(C_{ijkl})_{i,j,k,l=1}^N$  is a fourth-rank constitutive material tensor of the elastic medium which is physically referred to as the stiffness tensor; $\rho$ is a complex-valued function with the real part $\real \rho>0$ and imaginary part $\Ima \rho\geq 0$, denoting respectively the density and damping parameter of the elastic medium.
The stiffness tensor satisfies the following symmetries for a generic anisotropic elastic material:
\begin{equation}\label{eq:symmetry}
\mbox{major symmetry:}\quad C_{ijkl}=C_{klij},\qquad \mbox{minor symmetries:}\quad C_{ijkl}=C_{jikl}=C_{ijlk},
\end{equation}
for all $i,j,k,l=1,2, \cdots,N$. By Hooke's law,  the stress tensor $\sigma$ relates to the stiffness tensor $\mathcal{C}$ via the identity $\sigma(u):=\mathcal{C}:\nabla u$,
where the action of $\mathcal{C}$ on a matrix $A=(a_{ij})$ is defined as
\ben
\mathcal{C}:A=(\mathcal{C}:A)_{ij}=\displaystyle\sum_{k,l=1}^N C_{ijkl}\; a_{kl}.
\enn
Hence, the elliptic system in (\ref{I}) can be restated as
\be\label{Navier-C}
\nabla\cdot (\mathcal{C}: \nabla u)+\omega^2\rho u=0\quad\mbox{in}\quad \R^N.
\en
Note that in (\ref{I}) we have assumed the continuity of the \emph{stress vector} or \emph{traction}
( the normal component of the stress tensor)  on $\partial\Omega$, i.e.,
 $\mathcal{N}_{\mathcal{C}}^+u=\mathcal{N}_{\mathcal{C}}^-u$ where
\ben
\mathcal{N}_{\mathcal{C}}u:=\nu\cdot \sigma(u)=\left( \sum_{j,k,l=1}^N \nu_j C_{1jkl}  \frac{\partial u_k}{\partial x_l},\;  \sum_{j,k,l=1}^N \nu_j C_{2jkl}  \frac{\partial u_k}{\partial x_l},\, \cdots,\,
\sum_{j,k,l=1}^N \nu_j C_{Njkl}  \frac{\partial u_k}{\partial x_l}
  \right),
\enn
with $\nu=(\nu_1,\nu_2,\ldots,\nu_N)\textcolor{rot}{^\top}\in\mathbb{S}^{N-1}$ denoting the exterior unit normal vector to $\partial\Omega$ and $(\cdot)^\pm$ the limits taken from outside and inside of $\Omega$, respectively.

Since the elastic material in $\Omega^c$ is isotropic and homogeneous, one has
\be\label{C}
 C_{ijkl}(x)=\lambda \delta_{i,j}\delta_{k,l}+\mu ( \delta_{i,k}\delta_{j,l}+ \delta_{i,l}\delta_{j,k}),\quad x\in \Omega^c.
\en
That is, the stiffness tensor of the background medium is characterized by the Lam\'e constants $\lambda$ and $\mu$ which satisfy $\mu>0, N \lambda+2 \mu>0$.
Hence, the stress tensor in $\Omega^c$ takes the simple form
\ben
\sigma(u)=\lambda\, \textbf{I}\, \divv u + 2\mu\, \epsilon(u),\quad \epsilon(u):=\frac{1}{2}\left(\nabla u+\nabla u^\top\right),
\enn where $\textbf{I}$ stands for the $N\times N$ identity \textcolor{rot}{matrix}. Assuming that $\rho(x)\equiv \rho_0$ in $\Omega^c$,
the Lam\'e system (\ref{I}) reduces to the time-harmonic Navier equation
\be\label{eq:Naiver}
\Delta^*\,u+ \omega^2\rho_0  u = 0 \quad \mbox{in}\quad \Omega^c,\quad
 \Delta^*u:=\mu\Delta u+ (\lambda + \mu) \, \grad \divv\, u.
\en
Moreover, the surface traction $\mathcal{N}_{\mathcal{C}}u$ on $\partial \Omega$ takes the more explicit form
$
\mathcal{N}_{\mathcal{C}}\,u=T_{\lambda,\mu}u
$, where
\be\label{stress-2D}
T_{\lambda,\mu}u:=
2 \mu \, \partial_{\nu} u + \lambda \,
\nu \, \divv u
+\mu \nu^{\perp}\,(\partial_2 u_1 - \partial_1 u_2),\quad\,\nu=(\nu_1,\nu_2)\textcolor{rot}{^\top}, \;\nu^{\perp}:=(-\nu_2,\nu_1)\textcolor{rot}{^\top},
\en in two dimensions,
and
\be\label{stress-3D}
T_{\lambda,\mu}u:=2 \mu \, \partial_{\nu} u + \lambda \,
\nu \, \divv u+\mu \nu\times \curl u,\quad \nu=(\nu_1,\nu_2,\nu_3)\textcolor{rot}{^\top},
\en in three dimensions.
Here and also in what follows, we write $T_{\lambda,\mu}u=Tu$ to drop the dependance of $T_{\lambda,\mu}$ on the Lam\'e constants $\lambda$ and $\mu$ of the background medium.
Denote by
\ben
k_s := \omega\sqrt{\rho_0/\mu},\quad k_p = \omega \sqrt{\rho_0/(\lambda + 2\mu)}
\enn
the shear and compressional wave numbers of the background material, respectively.

Since the domain $\Omega^c$ is unbounded, an appropriate radiation condition at infinity must be imposed on $u^{sc}$ to ensure well-posedness of \textcolor{rot}{the} scattering problem. The scattered field in $\Omega^c$ can be decomposed into the sum of the compressional (longitudinal) part ${u}^{sc}_p$ and the shear (transversal) part ${u}^{sc}_s$ as follows (in three dimensions):
\be\label{decomposition}
u^{sc}={u}^{sc}_p+{u}^{sc}_s,\quad
{u}^{sc}_p=-\frac{1}{k_p^2}\,\mbox{grad}\,\mbox{div}\;{u}^{sc},\quad {u}^{sc}_s=\frac{1}{k_s^2}\,\mbox{curl}\,\mbox{curl}\;{u}^{sc}.
\en
In two dimensions, the shear part of the scattered field should be modified as
\be\label{decomposition-2d}
{u}^{sc}_s=\frac{1}{k_s^2}\,\overrightarrow{\mbox{curl}}\,\mbox{curl}\;{u}^{sc},
\en
where the two-dimensional operators \mbox{curl} and $\overrightarrow{\mbox{curl}}$ are defined respectively by
\ben
\mbox{curl}\,v=\partial_1 v_2-\partial_2 v_1,\quad v=(v_1,v_2)\textcolor{rot}{^\top},\qquad \overrightarrow{\mbox{curl}}\; f:=(\partial_2f, -\partial_1f)\textcolor{rot}{^\top}.
\enn
It then follows from the decompositions in (\ref{decomposition}) and (\ref{decomposition-2d}) that
\ben
(\Delta+k_\alpha^2)\, u_\alpha^{sc}=0,\qquad \alpha=p,s,\qquad \divv u_s^{sc}=0
\enn
and
\ben
 \mbox{curl}\,u_p^{sc}=0 \quad\mbox{in\quad 3D},\qquad \overrightarrow{\mbox{curl}}\,u_p^{sc}=0\quad\mbox{in\quad 2D}.
\enn
The scattered field is required to \textcolor{rot}{satisfy} the Kupradze radiation condition (see e.g. \cite{Kupradze})
\be
\label{RadiationCond}
\lim_{r \to \infty} r^{\frac{N-1}{2}}\left(\frac{\partial {u}^{sc}_p}{\partial r}-ik_p {u}_p^{sc}\right) = 0,
\; \lim_{r \to \infty} r^{\frac{N-1}{2}}\left(\frac{\partial {u}^{sc}_s}{\partial r}-ik_s {u}_s^{sc}\right) = 0,
\quad r=|x|
\en
uniformly with respect to all $\hat{{x}}={x}/|{x}|\in\mathbb{S}^{N-1}:=\{{x}\in\R^N:|{x}|=1\}$.
The radiation conditions in (\ref{RadiationCond}) lead to the P-part (longitudinal part) $u^{\infty}_p$ and the S-part (transversal part) $u^{\infty}_s$ of the far-field pattern of $u^{sc}$, given by the asymptotic behavior
\be\label{far}
u^{sc}(x)=\frac{\exp(ik_p|x|)}{|x|^{\frac{N-1}{2}}}\, u^{\infty}_p(\hat{x})+ \frac{\exp(ik_s|x|)}{|x|^{\frac{N-1}{2}}}\, u^{\infty}_s(\hat{x})+ \mathcal{O}(|x|^{-\frac{N+1}{2}}),\quad|x|\rightarrow+\infty,
\en
where, with some normalization, $u^{\infty}_p$ and $u^{\infty}_s$ are the far-field patterns of $u^{sc}_p$ and $u^{sc}_s$, respectively. \textcolor{rot}{We} define the far-field pattern $u^\infty$ of the scattered field $u^{sc}$ as the sum of $u^{\infty}_p$ and $u^{\infty}_s$, that is, $
 u^\infty:=u^{\infty}_p+u^{\infty}_s$.
Since $u^{\infty}_p$
 is normal to $\mathbb{S}^{N-1}$ and $u^{\infty}_s$ is  tangential to $\mathbb{S}^{N-1}$, it holds the relations
 \ben u^{\infty}_p(\hat{x})=(u^\infty(\hat{x})\cdot \hat{x})\,\hat{x},\quad
 u^{\infty}_s(\hat{x})=\left\{\begin{array}{lll}
 \hat{x}\times u^\infty(\hat{x})\times \hat{x}\quad &&\mbox{in}\quad 3D,\\
 (\hat{x}^\bot\cdot u^\infty(\hat{x}))\,\hat{x}^\bot\quad &&\mbox{in}\quad 2D.
 \end{array}\right.
\enn

Throughout this paper we make the following assumptions:
\begin{itemize}
\item[(A1)] There exists $R>0$ such that $\Omega\subset B_R:=\{x\in \R^N: |x|<R\}$ and that $u^{in}$ satisfies the Navier equation (\ref{eq:Naiver}) in $B_R$.

    \item[(A2)] The \textcolor{rot}{stiffness} tensor $\mathcal{C}$ satisfies the
uniform Legendre ellipticity condition
\be\label{LH}
\sum_{i,j,k,l=1}^N C_{ijkl}(x)\, a_{ij} a_{kl} \geq c_0 \sum_{i,j=1}^N |a_{ij}|^2,\quad a_{ij}=a_{ji},\quad c_0>0,
\en for all $x\in \Omega$.
In other words, $(\mathcal{C}(x):A):A\geq c_0 ||A||^2$ for all symmetry matrices $A=(a_{ij})^N_{i,j=1}\in \R^{N\times  N}$. Here $||A||$ means the Frobenius norm of the matrix $A$.
\item[(A3)] $||\rho||_{L^\infty(\Omega)}<\infty$,  and $||C_{ijkl}||_{L^\infty(\Omega)}<\infty$ for all $1\leq i,j,k,l\leq N$.
\end{itemize}

\begin{remark}
 The incident wave $u^{in}$ is allowed to be a linear combination of pressure and shear plane waves of the form
 \be\label{planewave}
 u^{in}(x,d)=\,c_p\,d \exp(ik_p x\cdot d)+c_s\,d^\bot\exp(ik_s x\cdot d),\quad c_p,c_s\in \C,
 \en with $d\in \mathbb{S}^{N-1}$ being the incident direction and $d^\bot\in \mathbb{S}^{N-1}$ satisfying $d^\bot\cdot d=0$.
  It also \textcolor{rot}{can be} elastic point source waves satisfying the equation
  \ben
  \Delta^*\,u^{in}(\cdot; y)+ \omega^2\rho_0  u^{in}(\cdot; y) = \delta(\cdot-y)\,\textbf{a} \quad \mbox{in}\quad \R^N\backslash\{y\},\quad
  \enn
where $y\in \R^N\backslash \overline{B}_R$ represents the location of the source and $\textbf{a}\in \C^N$ denotes the polarization direction. An explicit expression of $u^{in}(\cdot; y)$ is given by $u^{in}(\cdot; y)=\Pi(\cdot,y)\textbf{a}$ where $\Pi$ is the free-space Green's tensor to the Navier equation given by
\be\label{Pi}
\Pi(x,y)=\frac{1}{\mu}\Phi_{k_s}(x,y) \textbf{I}+\frac{1}{\rho_0\omega^2}\, \grad_x\,\grad_x^\top\;\left[\Phi_{k_s}(x,y)-\Phi_{k_p}(x,y) \right],\quad x\neq y.
\en
Here $\Phi_k$ ($k=k_p, k_s$) is the fundamental solution to the Helmholtz equation $(\Delta+k^2) u=0$ in $\R^N$.
It is well-known that
\be\label{source}
\Phi_k(x;y)=\left\{\begin{array}{lll}
\frac{i}{4}H_0^{(1)}(k|x-y|),&& N=2,\\
\frac{e^{ik|x-y|}}{4\pi|x-y|},&& N=3,
\end{array}\qquad x\neq y,
\right.
\en with $H_0^{(1)}(\cdot)$ being the Hankel function of the first kind of order zero.
\end{remark}

Let $H^1(B_R)$ denote the Sobolev space of scalar functions on $B_R$. In the following
 we state the uniqueness and existence of weak solutions to our scattering problem
 in the energy space $X_R:=(H^1(B_R))^N$.
 \begin{theorem}\label{theorem}
 Under the assumptions (A1)-(A3) there exists a unique solution $u\in X_R$ to the scattering problem (\ref{I}), (\ref{eq:Naiver}) and (\ref{RadiationCond}).
 \end{theorem}
 The proof of Theorem \ref{theorem} will depend on the Fredholm alternative together with properties of the Dirichlet-to-Neumann mapping on $\Gamma_R\textcolor{rot}{:=\partial B_R}$. As a consequence, we also obtain the well-posedness of the scattering problem due to an impenetrable elastic body with various kinds of boundary conditions.

\begin{corollary}\label{corollary}
Consider the time-harmonic elastic scattering from an impenetrable bounded elastic body $\Omega$ with Lipschitz boundary embedded \textcolor{rot}{in} a homogeneous isotropic medium. Suppose that the total field satisfies
one of the following boundary conditions \textcolor{rot}{on} $\partial \Omega$:
\begin{description}
\item[(i)] The first kind (Dirichlet)  boundary condition: $u=0$;
\item[(ii)] The second kind (Neumann) boundary condition: $Tu=0$;
\item[(iii)] The third kind  boundary condition: $\textcolor{rot}{\nu\cdot u=0, \nu\times Tu=0}$ in 3D, $\nu\cdot u=\nu^\bot\cdot Tu=0$ in 2D;
\item[(iv)] The fourth kind  boundary condition: $\textcolor{rot}{\nu\times u=0, \nu\cdot Tu=0}$ in 3D, $\nu^\bot\cdot u=\nu\cdot Tu=0$ in 2D;
\item[(v)] Robin boundary condition: $Tu-i\eta u=0$,\quad $\eta\in\C$, $\mbox{Re}(\eta)>0$.
\end{description}
 Then the scattered field $u^{sc}=u-u^{in}$ is uniquely solvable in $(H^1_{loc}(\R^N\backslash\overline{\Omega}))^N$.
\end{corollary}

The variational approach for proving Theorem \ref{theorem} can be easily adapted to treat the boundary value problems in Corollary \ref{corollary}. We omit the details for simplicity and refer to \cite{EH} for the proof in unbounded periodic structures.
\begin{remark} Using integral equation methods,
well-posedness of the boundary value problems in Corollary \ref{corollary} has been investigated in Kupradze \cite{Kupradze,TC} for scatterers with $C^2$-smooth boundaries. The variational arguments presented here have thus relaxed the regularity of the boundary to be Lipschitz.
\end{remark}

\subsection{Variational formulation with transparent boundary operator}
Let $R>0$ be specified in assumption (A1). \textcolor{rot}{By} the first Betti\textcolor{rot}{'s} formula, it follows that for $u, v\in X_R$,
\be\label{Betti}
-\int_{B_R}[\nabla\cdot (\mathcal{C}: \nabla u)+\omega^2\rho u]\cdot\overline{v}\,dx=
\int_{B_R}[(\mathcal{C}: \nabla u) :\nabla \overline{v}-\omega^2\rho\; u\cdot\overline{v}]\,dx-\int_{\Gamma_R} Tu\cdot\overline{v}\,ds.
\en
Below we introduce the Dirichlet-to-Neumann (DtN) map in a homogeneous isotropic background medium, allowing us to reduce the scattering problem \textcolor{rot}{on a bounded domain}.
\begin{definition}
For any $w\in \textcolor{rot}{(H^{1/2}(\G_R))^N}$,  the DtN map $\mathcal{T}$ acting on $w$ is defined as $$\mathcal{T}w:=(Tv^{sc})|_{\G_R},$$
where $v^{sc}\in \textcolor{rot}{(H^1_{loc}(\R^N\backslash\overline{B}_R))^N}$ is the unique radiating solution to the boundary value problem
\be
\label{DtN}
\Delta^{*}v^{sc} +   \omega^2\rho_0v^{sc} = 0\quad\mbox{in}\quad \R^N\backslash\ov{B_R},\qquad
v^{sc} = w \quad\mbox{on}\quad \G_R.
\en
\end{definition}
\begin{remark}
The DtN map $\mathcal{T}$ is well-defined, since the Dirichlet-kind boundary value problem (\ref{DtN}) is uniquely solvable in $\textcolor{rot}{(H^1_{loc}(\R^N\backslash\overline {B_R}))^N}$; see Remarks \ref{Remark-2} and \ref{Remark-3} for the explicit expressions in terms of special functions.
\end{remark}

To obtain an equivalent variational formulation of (\ref{I}), we shall apply Betti's identity (\ref{Betti}) to a solution $u=u^{in}+u^{sc}$ in $B_R$ and use the relation
\ben
Tu=Tu^{in}+Tu^{sc}=Tu^{in}+\mathcal{T}u^{sc}=f+\mathcal{T}u,\quad f:=(Tu^{in}-\mathcal{T}u^{in})|_{\Gamma_R}.
\enn
Then the variational formulation
reads as follows: find $u=(u_1,\cdots,u_N)\in X_R$ such that
\be\label{variational}
a(u,v)=\int_{\G_R} f\cdot\overline{v}\,ds\qquad\qquad\mbox{for all}\quad v=(v_1,v_2,\cdots,v_N)\textcolor{rot}{^\top}\in X_R,
\en
where the sesquilinear form $a(\cdot,\cdot): X_R\times X_R\rightarrow \C $ is defined by
\be\label{sesquilinear}
a(u,v):=\int_{B_R}\left\{  \displaystyle\sum_{i,j,k,l=1}^N  C_{ijkl}  \frac{\partial u_k}{\partial x_l} \frac{\partial \overline{v_i}}{\partial x_j}
 -\omega^2\,\rho\, u_i\overline{v}_i\right\} dx-\int_{\G_R} \mathcal{T}u\cdot\overline{v}\,ds.
\en
\begin{remark}
The variational problem (\ref{variational}) and \textcolor{rot}{the} scattering problem (\ref{I}), (\ref{eq:Naiver}), (\ref{RadiationCond})
are equivalent in the following sense. If $u^{sc}\in (H^1_{loc}(\R^N))^N$ is a solution of \textcolor{rot}{the} scattering problem \textcolor{rot}{(\ref{I}), (\ref{eq:Naiver}) and (\ref{RadiationCond})}, then
the restriction of the total field $u$ to $B_R$, i.e., $u|_{B_R}$, satisfies the variational problem (\ref{variational}). Conversely, a solution $u\in X_R$ of (\ref{variational}) can be extended to a solution $u=u^{in}+u^{sc}$ of the Lam\'e system in $|x|>R$, where $u^{sc}$ is defined as the unique radiating solution to the isotropic Lam\'e system in $|x|>R$ satisfying the Dirichlet boundary value $u^{sc}=u-u^{in}$ on $\Gamma_R$.
\end{remark}
In the following lemma, we show
properties of the DtN map $\mathcal{T}$ which play an essential role in our uniqueness and existence proofs. The two and three dimensional proofs will be carried out \textcolor{rot}{in} the subsequent Sections \ref{prof-dtn-2d} and \ref{prof-dtn-3d}\textcolor{rot}{, respectively}.
\begin{lemma}\label{DtN-lemma}
\begin{description}
\item[(i)] $\mathcal{T}$ is a bounded operator from $(H^{1/2}(\Gamma_R))^N$ to $(H^{-1/2}(\Gamma_R))^N$.
\item[(ii)] The operator $-\mathcal{T}$ can be decomposed into the sum of a positive operator $\textcolor{rot}{\mathcal{T}_1}$ and a compact operator $\mathcal{T}_2$, that is, $-\mathcal{T}=\mathcal{T}_1+\mathcal{T}_2$ on $(H^{1/2}(\Gamma_R))^N$.
\end{description}
\end{lemma}
Let $X_R'$ denote the dual of $X_R$ with respect to the inner product of $(L^2(B_R))^N$. By the boundedness of $\rho$,
$C_{ijkl}$ (see Assumption (A3)) and $\mathcal{T}$,  there exists a continuous linear operator $\mathcal{A}: X_R\rightarrow X_R'$ associated with the sesquilinear form $a$ such that
\be\label{A}
a(u,v)=<\mathcal{A}u,v>\qquad\mbox{for all}\quad v\in X_R.
\en Here and henceforth the notation $<\cdot,\cdot>$ denotes the duality between $X_R'$ and $X_R$.
By Assumption (A1) and Lemma \ref{DtN-lemma} (ii), there exists  $\mathcal{F}\in X_R'$ such that
\ben
\int_{\G_R} f\cdot\overline{v}\,ds=<\mathcal{F},v>\qquad\mbox{for all}\quad v\in X_R.
\enn
Hence the variational formulation (\ref{variational}) can be written as an operator equation \textcolor{rot}{of} finding $u\in X_R$ such that
\ben
\mathcal{A}\,u=\mathcal{F}\qquad\mbox{in}\quad X_R'.
\enn
Below we recall the definition of strong ellipticity.
\begin{definition}
A bounded sesquilinear form $a(\cdot,\cdot)$ \textcolor{rot}{on} some Hilbert space $X$ is called strongly elliptic if there
exists a compact form $q(\cdot,\cdot)$ such that
\ben
|\real\, a(u,u)|\geq C\,||u||^2_{X}-q(u,u) \qquad\mbox{for all}\quad u\in X.
\enn
\end{definition}

The following theorem establishes the strong ellipticity of the sesquilinear form $a$ defined \textcolor{rot}{by} (\ref{sesquilinear}).
\begin{theorem}\label{strong-elliptic}
The sesquilinear form $a(\cdot,\cdot)$ is strongly elliptic over $X_R$ under the Assumption (A2).
Moreover, the operator $\mathcal{A}: X_R\rightarrow X_R'$ defined by (\ref{A}) \textcolor{rot}{is} a Fredholm operator with index zero.
\end{theorem}
\begin{proof} We may rewrite the form $a$ as the sum $a=a_1+a_2$, where
the sesquilinear forms $a_j$ ($j=1,2$) are defined as
\ben
a_1(u,v)&:=&\int_{B_R}\left\{  \displaystyle\sum_{i,j,k,l=1}^N  C_{ijkl}  \frac{\partial u_k}{\partial x_l} \frac{\partial \overline{v_i}}{\partial x_j}\right\} dx+\int_{\G_R} \mathcal{T}_1u\cdot\overline{v}\,ds,\\
a_2(u,v)&:=& -\omega^2\int_{B_R} u\cdot\overline{v}\,dx+\int_{\G_R} \mathcal{T}_2u\cdot\overline{v}\,ds,
\enn
Note that $\mathcal{T}_j$ ($j=1,2$) are the operators given by Lemma \ref{DtN-lemma}. It is seen from the uniform Legendre ellipticity condition and Lemma \ref{DtN-lemma} (ii) that $a_1$ is coercive over $X_R$. The compact embedding of $X_R$ into $(L^2(B_R))^N$ and the compactness of $\mathcal{T}_2$ give the compactness of the form $a_2$. Hence
$a(\cdot,\cdot)$ is strongly elliptic over $X_R$ and thus $\mathcal{A}$ is a Fredholm operator with index zero.
\end{proof}

{\bf Proof of Theorem \ref{theorem}.} Using Theorem \ref{strong-elliptic} and applying the Fredholm alternative, we only need to prove the uniqueness of our scattering problem. Letting $u^{in}\equiv 0$ (which implies that $f=0$ in $X_R'$) and taking the imaginary part of (\ref{variational}) with $v=u^{sc}$ we get
\ben
\Ima \int_{\G_R} \mathcal{T}u^{sc}\cdot\overline{u^{sc}}\,ds=0.
\enn
By the analogue of Rellich's lemma in elasticity (see Lemmas \ref{lem:Rellich-2D} and \ref{Lem:generalizedDtN-3D} below) we obtain $u^{sc}\equiv0$ in $B_R$. This proves the uniqueness and Theorem \ref{theorem}.
\hfill $\Box$

The remaining part of this section will be devoted to the proof of properties of the DtN map in a more general setting.  We shall consider the \emph{generalized} stress vector (cf. (\ref{stress-2D}))
\be
\label{GStress}
\widetilde{T}_{\widetilde{\lambda},\widetilde{\mu}}u:=\left\{\begin{array}{lll}
(\mu+\widetilde{\mu})\nu\cdot\mbox{grad}\,u + \widetilde{\lambda}\nu\,\mbox{div}\,u - \widetilde{\mu}\nu^\perp\mbox{curl}\,u,\quad &&\mbox{if}\quad N=2,\\
(\mu+\widetilde{\mu})\nu\cdot\mbox{grad}\,u + \widetilde{\lambda}\nu\,\mbox{div}\,u + \widetilde{\mu}\nu\times\mbox{curl}\,u,\quad &&\mbox{if}\quad N=3,
\end{array}\right.
\en
where $\widetilde{\lambda}$, $\widetilde{\mu}\in \R$ satisfying
$\widetilde{\lambda}+\widetilde{\mu}=\lambda+\mu$. In the present paper we suppose that
\be\label{eq:7}
\frac{(\lambda-\mu)(\lambda+2\mu)}{\lambda+3\mu}<\tilde{\lambda}<\lambda+2\mu.
\en
The assumption (\ref{eq:7}) will be used later for proving Lemma \ref{Lem:generalizedDtN} (ii) and Lemma \ref{Lem:generalizedDtN-3D} (ii). We emphasize that the above condition (\ref{eq:7}) covers at least the following three cases:
\begin{description}
\item[Case (i):] $\tilde{\lambda}=\lambda$, $\tilde{\mu}=\mu$.

 \item[Case (ii):] $\tilde{\lambda}=\lambda+\mu$, $\tilde{\mu}=0$.
 \item[Case (iii):] $\tilde{\lambda}=(\lambda+2\mu)(\lambda+\mu)/(\lambda+3\mu)$, $\tilde{\mu}=\mu(\lambda+\mu)/(\lambda+3\mu)$.
\end{description}
 Note that the usual surface traction  coincides with $\widetilde{T}_{\widetilde{\lambda}, \widetilde{\mu}}$ in the case (i). Properties of the DtN map in case (ii) were analyzed in \cite{EH2012} on a line and in \cite{Li2016} on a circle.

  The generalized DtN map $\widetilde{\mathcal{T}}$
corresponding to (\ref{GStress}) is defined as
 \ben
 \widetilde{\mathcal{T}} w=(\widetilde{T}_{\widetilde{\lambda},\widetilde{\mu}}\, v^{sc})|_{\Gamma_R},\qquad w\in \textcolor{rot}{(H^{1/2}(\Gamma_R))^N},
 \enn
 where $v^{sc}\in \textcolor{rot}{(H^1_{loc}(\Omega^c))^N}$ is the radiating solution to the isotropic homogeneous Navier equation (\ref{eq:Naiver}) in $|x|\geq R$.

\subsection{Properties of DtN \textcolor{rot}{map} in 2D}\label{prof-dtn-2d}

In this section we verify Lemma \ref{DtN-lemma} and the Rellich's identity for the generalized DtN map $\widetilde{\mathcal{T}}$ in $\R^2$. For this purpose, the surface vector harmonics in $\R^2$ are needed.
Denote by $(r,\theta_x)$ the
 polar coordinates of $x=(x_1,x_2)\textcolor{rot}{^\top}\in \R^2$, and by $\hat{\boldsymbol{r}}$, $\hat{\boldsymbol{\theta}}\in \mathbb{S}^{1}$ the unit vectors under the polar coordinates \textcolor{rot}{such that}
 \ben
 \hat{\boldsymbol{r}}=(\cos\theta,\sin\theta)\textcolor{rot}{^\top},\qquad \hat{\boldsymbol{\theta}}=(-\sin\theta,\cos\theta)\textcolor{rot}{^\top},\quad \theta\in[0,2\pi).
 \enn
Let $\bb{P}_{n}$ and $\bb{S}_{n}$ be the surface vector harmonics in two-dimensions defined as
\be
\label{vsh1}
\bb{P}_{n}(\hat{\bb{x}}):= e^{in\theta_x}\hat{\boldsymbol{r}},\quad
\bb{S}_{n}(\hat{\bb{x}}):= e^{in\theta_x}\hat{\boldsymbol{\theta}},\quad \hat{\bb{x}}=x/|x|\in \mathbb{S}^1.
\en

Below we shall derive a series representation of the generalized DtN map.
The solution $v^{sc}$ can be split into the sum of a pressure part with vanishing curl and a shear part with vanishing divergence, that is,
\be\label{radiation-solution}
v^{sc}=\mbox{grad}\,\psi_p+\overrightarrow{\mbox{curl}}\,{\psi}_s\quad\mbox{in}\quad |x|\geq R,
\en
where $\psi_p$ and $\psi_s$ are both scalar functions.
It then follows that
\be
\label{pre2}
&&\Delta\psi_\alpha+k_\alpha^2 \psi_\alpha =0,\quad
\lim_{r \to \infty} r^{1/2}\left(\frac{\partial \psi_\alpha}{\partial r}-ik_\alpha\psi_\alpha\right) = 0,\quad \alpha=p,s.
\en
The solutions of (\ref{pre2}) can be expressed as
\be\label{radiation-solution-ps}
\psi_\alpha(x)=\sum_{n=-\infty}^\infty \frac{H_n^{(1)}(k_\alpha r)}{H_n^{(1)}(k_\alpha R)}\psi_\alpha^{n}\; e^{in\theta_x},\quad r=|x|\ge R,\quad \alpha=p,s,
\en
where $\psi_\alpha^{n}\in \C$  stand for the Fourier coefficients of $\psi_\alpha|_{\Gamma_R}$ and $H_n^{(1)}$ is the Hankel function of the first kind of order $n$.
Set
\ben
 t_\alpha:=k_\alpha R,\quad
\gamma_\alpha:=\frac{{H_n^{(1)}}'(t_\alpha)}{H_n^{(1)}(t_\alpha)},\quad
\beta_\alpha:=\frac{{H_n^{(1)}}''(t_\alpha)}{H_n^{(1)}(t_\alpha)},\qquad \alpha=p,s.
\enn
 Let $(\cdot,\cdot)$ be the $L^2$ inner product on the \textcolor{rot}{unit circle} given by
\ben
(u,v):=\frac{1}{2\pi}\int_0^{2\pi} u\cdot\overline{v}\,d\theta\quad \mbox{for all}\quad u,v\in \textcolor{rot}{(L^2(\mathbb{S}^1))^2}.
\enn
Due to the orthogonality relations between $\bb{P}_{n}$ and $\bb{S}_{n}$, it is easy to derive from
(\ref{radiation-solution}) and (\ref{radiation-solution-ps})
 that
\ben
\left(v^{sc}|_{\G_R},\bb{P}_{n}\right) = \frac{1}{R} \left[t_p\gamma_p{\psi}_p^{n}+in\psi_s^{n}\right],\quad
\left(v^{sc}|_{\G_R},\bb{S}_{n}\right) = \frac{1}{R} \left[in\psi_p^{n}-t_s\gamma_s{\psi}_s^{n}\right].
\enn
Equivalently, the previous relations can be written in the matrix form
\be
\label{series-DV}
A_n\begin{bmatrix}
\psi_p^{n} \\
\psi_s^{n}
\end{bmatrix}
=R\begin{bmatrix}
\left(v^{sc}|_{\G_R},\bb{P}_{n}\right) \\
\left(v^{sc}|_{\G_R},\bb{S}_{n}\right)
\end{bmatrix}
,\quad
A_n:=\begin{bmatrix}
t_p\gamma_p & in \\
in & -t_s\gamma_s
\end{bmatrix}.
\en

\begin{lemma}\label{lem:An}
The matrix $A_n$ is  invertible for all $n\in\Z$ and $R>0$. Its inverse is given by
\be
\label{MatrixIA}
A_n^{-1}=\frac{1}{\Lambda_n}\begin{bmatrix}
-t_s\gamma_s & -in \\
-in & t_p\gamma_p
\end{bmatrix},
\qquad
\Lambda_n:=\mbox{det}(A_n)=n^2-t_pt_s\gamma_p\gamma_s.
\en
\end{lemma}
\begin{proof}
It's sufficient to prove that $\Lambda_n\ne 0$. We write $\Lambda_n$ as
\ben
\Lambda_n=n^2-I_n(t_P)\;I_n(t_s),\quad
I_n(z):=z{H_n^{(1)}}'(z)/H_n^{(1)}(z).
\enn
Making use of the  Wronskian identity for Bessel and Neumann functions (see, e.g., \cite[Chapter 3.4]{CK98}), it is easy to derive that
\ben
\Ima(I_n(z))=\frac{2}{\pi|H_n^{(1)}(z)|^2}\quad\mbox{for all}\quad n\in \Z,\;z>0.
\enn This implies that, for any fixed $n\in \Z$,
\ben
\real (\Lambda_n)&=&n^2-\real(I_n(t_P))\;\real(I_n(t_s))+\Ima(I_n(t_P))\;\Ima(I_n(t_s)),\\
\Ima (\Lambda_n)&=&-\real(I_n(t_P))\;\Ima(I_n(t_s))-\Ima(I_n(t_P))\;\real(I_n(t_s))
\enn
cannot vanish simultaneously. Hence, $\Lambda_n\ne 0$.
\end{proof}
\begin{remark}\label{Remark-2}
The unique radiating solution $v^{sc}$ to the boundary value problem (\ref{DtN}) can be represented as the series (\ref{radiation-solution}) and (\ref{radiation-solution-ps}), where the coefficients $\psi_p^{n}$ and
    $\psi_s^{n}$ are given by
    \ben
\begin{bmatrix}
\psi_p^{n} \\
\psi_s^{n}
\end{bmatrix}
=R A_n^{-1}\begin{bmatrix}
\left(w,\bb{P}_{n}\right) \\
\left(w,\bb{S}_{n}\right)
\end{bmatrix}.
    \enn
\end{remark}
Now, we turn to investigating
 the generalized stress vector (cf. (\ref{GStress}))
\ben
\widetilde{\mathcal{T}}v^{sc}=(\mu+\widetilde{\mu})\,\hat{\bb{r}}\cdot\mbox{grad}\,v^{sc} + \widetilde{\lambda}\,\hat{\bb{r}}\;\mbox{div}\,v^{sc} - \widetilde{\mu}\,\hat{\bb{\theta}}\;\mbox{curl}\,v^{sc}\qquad \mbox{on}\quad |x|=R.
\enn Inserting (\ref{radiation-solution}) into the previous identity and using the relations
\ben
\mbox{grad}=\hat{\bb{r}}\frac{\partial}{\partial r}+\frac{1}{r} \hat{\bb{\theta}}\frac{\partial}{\partial \theta},\qquad
\hat{\bb{r}}\cdot \mbox{grad}=\frac{\partial}{\partial r},\quad
\overrightarrow{\mbox{curl}}=-\hat{\bb{\theta}}\frac{\partial}{\partial r}+\frac{1}{r} \hat{\bb{r}}\frac{\partial}{\partial \theta},
\enn
 we obtain via straightforward calculations that
\ben
\hat{\bb{r}}\cdot\widetilde{\mathcal{T}}v^{sc} &=&
(\mu+\widetilde{\mu})\, \hat{\bb{r}}\cdot \frac{\partial}{\partial r}\left[
\hat{\bb{r}}\frac{\partial \psi_p}{\partial r}+\frac{1}{r} \hat{\bb{\theta}}\frac{\partial \psi_p}{\partial \theta}
-\hat{\bb{\theta}}\frac{\partial \psi_s}{\partial r}+\frac{1}{r} \hat{\bb{r}}\frac{\partial \psi_s}{\partial \theta}
\right]+\tilde{\lambda}\,\mbox{div}\,\mbox{curl}\psi_p\\
&=&(\mu+\widetilde{\mu}) \left(\frac{\pa^2 \psi_p}{\pa r^2}-\frac{1}{r^2}\frac{\pa \psi_s}{\pa \theta}+\frac{1}{r}\frac{\pa^2 \psi_s}{\pa r\pa \theta}\right) +\widetilde{\lambda}\Delta\psi_p,\\
\hat{\bb{\theta}}\cdot\widetilde{\mathcal{T}}v^{sc}
&=&(\mu+\widetilde{\mu})\, \hat{\bb{\theta}}\cdot \frac{\partial}{\partial r}\left[
\hat{\bb{r}}\frac{\partial \psi_p}{\partial r}+\frac{1}{r} \hat{\bb{\theta}}\frac{\partial \psi_p}{\partial \theta}
-\hat{\bb{\theta}}\frac{\partial \psi_s}{\partial r}+\frac{1}{r} \hat{\bb{r}}\frac{\partial \psi_s}{\partial \theta}
\right]-\tilde{\mu}\,\mbox{curl}\,\overrightarrow{\mbox{curl}} \psi_s\\
&=& (\mu+\widetilde{\mu}) \left(-\frac{1}{r^2}\frac{\pa \psi_p}{\pa \theta}+\frac{1}{r}\frac{\pa^2 \psi_p}{\pa r\pa \theta}-\frac{\pa^2 \psi_s}{\pa r^2}\right) +\widetilde{\mu}\Delta\psi_s.
\enn
This implies that
\be\label{DtN-matix}
\begin{bmatrix}
\left(\widetilde{\mathcal{T}}v^{sc}|_{\G_R},\bb{P}_{n}\right) \\
\left(\widetilde{\mathcal{T}}v^{sc}|_{\G_R},\bb{S}_{n}\right)
\end{bmatrix}
=\frac{1}{R^2}\,B_n \begin{bmatrix}
\psi_p^n \\ \psi_s^n
\end{bmatrix}
\en
where
\be
\label{MatrixB}
B_n:=\begin{bmatrix}
(\mu+\widetilde{\mu})t_p^2\beta_p -\widetilde{\lambda}t_p^2 & i(\mu+\widetilde{\mu})n\left( t_s\gamma_s-1\right) \\
i(\mu+\widetilde{\mu})n\left(t_p\gamma_p-1\right) & -\left(\mu+\widetilde{\mu}\right)t_s^2\beta_s -\widetilde{\mu}t_s^2
\end{bmatrix}.
\en
Combining (\ref{DtN-matix}) with (\ref{series-DV}) gives the relation 
\be
\label{TractionR}
\begin{bmatrix}
\left(\widetilde{\mathcal{T}}(v^{sc}|_{\G_R}),\bb{P}_{n}\right) \\
\left(\widetilde{\mathcal{T}}(v^{sc}|_{\G_R}),\bb{S}_{n}\right)
\end{bmatrix}
=W_n\begin{bmatrix}
\left(v^{sc}|_{\G_R},\bb{P}_{n}\right) \\
\left(v^{sc}|_{\G_R},\bb{S}_{n}\right)
\end{bmatrix},\qquad
W_n:=\frac{1}{R}B_nA_n^{-1}.
\en

Properties of the two-dimensional DtN \textcolor{rot}{map} are summarized in the subsequent two lemmas.

\begin{lemma}\label{Lem:generalizedDtN} Let $w=\sum_{n\in \Z} w_p^n \bb{P}_n+w_s^n \bb{S}_n\in \textcolor{rot}{(H^{1/2}(\Gamma_R))^2}$. Then,
 \begin{description}
\item[(i)]The generalized DtN operator $\widetilde{\mathcal{T}}$ takes the form
\ben
\widetilde{\mathcal{T}} w=\sum_{n\in\Z} W_n \begin{bmatrix}w_p^n\\ w_s^n \end{bmatrix}
\enn in the orthogonal basis $\{(\bb{P}_n, \bb{S}_n): n\in \Z\}$. Moreover, $\widetilde{\mathcal{T}}$ is a bounded linear \textcolor{rot}{operator} from $\textcolor{rot}{(H^{s}(\Gamma_R))^2}$  to $\textcolor{rot}{(H^{s-1}(\Gamma_R))^2}$ for all $s\in \R$.
\item[(ii)] For sufficiently large $M>0$, the real part of the operator
\ben
-\widetilde{\mathcal{T}_1} w:=-\sum_{|n|\geq M} W_n \begin{bmatrix}w_p^n\\ w_s^n \end{bmatrix}
\enn is positive over $\textcolor{rot}{(H^{1/2}(\Gamma_R))^2}$, and $\widetilde{\mathcal{T}}-\widetilde{\mathcal{T}_1}$ is a compact operator.
\end{description}
\end{lemma}
\begin{proof}
(i) We only need to show the boundedness of $\widetilde{\mathcal{T}}$. Recall that
\ben
||w||_{\textcolor{rot}{(H^{s}(\Gamma_R))^2}}&=&\left(\sum_{n\in \Z}(1+|n|)^{2s} |w^n|^{2}\right)^{1/2},\quad
w^n:=[w_p^n, w_s^n]^\top,\\
||\widetilde{\mathcal{T}}w||_{\textcolor{rot}{(H^{s-1}(\Gamma_R))^2}}&=&\left(\sum_{n\in \Z}(1+|n|)^{2(s-1)} |W_nw^n|^{2}\right)^{1/2}.
\enn
Hence, it suffices to estimate the max norm of the matrix $W_n$ bounded by
\be\label{maxnorm}
||W_n||_{\mbox{max}}\leq C\,|n|,
\en
for some constant $C>0$ uniformly in all $n\in \Z$, so that $|W_n w^n|^2\leq C^2\,|n|^2 |w^n|^2$.

It holds that
\ben
{H_n^{(1)}}''(z) &=& \left(H_{n-1}^{(1)}(z)-\frac{n}{z} H_n^{(1)}(z)\right)'\\
&=& -H_n^{(1)}(z)+\frac{n-1}{z}H_{n-1}^{(1)}(z) +\frac{n}{z^2}H_n^{(1)}(z)-\frac{n}{z}\left(H_{n-1}^{(1)}(z) -\frac{n}{z} H_n^{(1)}(z)\right)\\
&=& \frac{n^2+n-z^2}{z^2}H_n^{(1)}(z)-\frac{1}{z} H_{n-1}^{(1)}(z)\\
&=& \frac{n^2+n-z^2}{z^2}H_n^{(1)}(z)-\frac{1}{z} \left({H_n^{(1)}}'(z)+\frac{n}{z} H_n^{(1)}(z)\right)\\
&=& \left(\frac{n^2}{z^2}-1\right)H_n^{(1)}(z)- \frac{1}{z} {H_n^{(1)}}'(z),
\enn
giving rise to the identities
\be\label{eq:4}
\beta_p=\frac{n^2}{t_p^2}-1-\frac{1}{t_p}\gamma_p, \quad \beta_s=\frac{n^2}{t_s^2}-1-\frac{1}{t_s}\gamma_s.
\en
From the expressions of $A_n^{-1}$ and $B_n$ we get the entries $W_n^{(i,j)}$ of $W_n$, given by
\ben
W_n^{(1,1)} &=& \frac{1}{R\Lambda_n}\left\{-t_s\gamma_s\left[ (\mu+\widetilde{\mu})t_p^2\beta_p -\widetilde{\lambda}t_p^2\right] +n^2(\mu+\widetilde{\mu})(t_s\gamma_s-1)\right\}\\
&=& \frac{1}{R\Lambda_n}\left[ -(\mu+\widetilde{\mu})\Lambda_n +\omega^2\rho_0R^2t_s\gamma_s\right], \\
W_n^{(2,2)} &=& \frac{1}{R\Lambda_n}\left\{ n^2(\mu+\widetilde{\mu})(t_p\gamma_p-1) -t_p\gamma_p\left[ (\mu+\widetilde{\mu})t_s^2\beta_s +\widetilde{\mu}t_s^2\right]\right\}\\
&=& \frac{1}{R\Lambda_n}\left[-(\mu+\widetilde{\mu})\,\Lambda_n +\omega^2\rho_0R^2t_p\gamma_p\right],\\
W_n^{(1,2)} &=& \frac{1}{R\Lambda_n}\left\{-in\left[ (\mu+\widetilde{\mu})t_p^2\beta_p -\widetilde{\lambda}t_p^2\right] +int_p\gamma_p(\mu+\widetilde{\mu})(t_s\gamma_s-1)\right\} \\
&=& \frac{1}{R\Lambda_n}\left[ -in(\mu+\widetilde{\mu})\,\Lambda_n +in\omega^2\rho_0R^2\right],\\
W_n^{(2,1)} &=& \frac{1}{R\Lambda_n}\left\{ -in(\mu+\widetilde{\mu})t_s\gamma_s(t_p\gamma_p-1) +in\left[ (\mu+\widetilde{\mu})t_s^2\beta_s +\widetilde{\mu}t_s^2\right]\right\} \\
&=& \frac{1}{R\Lambda_n}\left[ in(\mu+\widetilde{\mu})\,\Lambda_n -in\omega^2\rho_0R^2\right],
\enn
in which we have used (\ref{eq:4}) and the fact that
$
\widetilde{\lambda}+\widetilde{\mu}=\lambda+\mu.
$

From the series expansions of the Bessel and Neumann functions (see, e.g., \cite[Chapter 3]{CK98}) we know
\ben
H_n^{(1)}(z)=\frac{(n-2)!}{i\pi}\left(\frac{2}{z}\right)^{n-1} \left[(n-1)\left(\frac{2}{z}\right)^2+1+O\left(\frac{1}{n}\right)\right],\quad n\rightarrow +\infty.
\enn
This implies that
\be\no
\frac{H_{n-1}^{(1)}(z)}{H_n^{(1)}(z)} &=& \frac{z}{2n-4} \frac{1+(n-2)\left(\frac{2}{z}\right)^2+O\left(\frac{1}{n}\right)} {1+(n-1)\left(\frac{2}{z}\right)^2+O\left(\frac{1}{n}\right)}\\ \no
&=& \left[\frac{z}{2n}+O\left(\frac{1}{n^2}\right)\right] \left[1+O\left(\frac{1}{n}\right)\right] \\ \label{eq:1}
&=& \frac{z}{2n} +O\left(\frac{1}{n^2}\right).
\en
The asymptotic behavior (\ref{eq:1}) together with the relation ${H_n^{(1)}}'=-n/z H_n^{(1)}+H_{n-1}^{(1)}$
leads to
\ben
\frac{{H_n^{(1)}}'(z)}{H_n^{(1)}(z)}=  -\frac{n}{z}+\frac{z}{2n}+O\left(\frac{1}{n^2}\right),\quad n\rightarrow +\infty.
\enn
Since $H_{-n}^{(1)}(z)=(-1)^nH_n^{(1)}(z)$, we obtain as $|n|\rightarrow\infty$ that
\be\label{eq:2}
&&\gamma_\alpha=\frac{{H_n^{(1)}}'(t_\alpha)}{H_n^{(1)}(t_\alpha)}= -\frac{|n|}{t_\alpha}+\frac{ t_\alpha }{2|n|}+O\left(\frac{1}{n^2}\right),\quad \alpha=p,s,
\\ \label{eq:3}
&&\Lambda_n =\frac{R^2(k_p^2+k_s^2)}{2}+O(\frac{1}{|n|})=\frac{R^2\rho_0\omega^2(\lambda+3\mu)}{2\mu(\lambda+2\mu)} +O\left(\frac{1}{|n|}\right).
\en
Inserting (\ref{eq:2}) and (\ref{eq:3}) into the expression of $W_n^{(i,j)}$ yields

\ben
W_n^{(1,1)} &=& -\frac{2\mu(\lambda+2\mu)} {R(\lambda+3\mu)}|n|+O(1),\\
W_n^{(2,2)} &=& -\frac{2\mu(\lambda+2\mu)} {R(\lambda+3\mu)}|n|+O(1),\\
W_n^{(1,2)} &=& \frac{i\left[(\mu+\widetilde{\mu}) (\lambda+3\mu)-2\mu(\lambda+2\mu) \right]}{R(\lambda+3\mu)}|n| +O(1),\\
W_n^{(2,1)} &=& -\frac{i\left[(\mu+\widetilde{\mu}) (\lambda+3\mu)-2\mu(\lambda+2\mu) \right]}{R(\lambda+3\mu)}|n| +O(1),
\enn
from which the
estimate (\ref{maxnorm}) follows directly.

(ii) Define $\widetilde{W}_n:=-(W_n+W_n^*)/2$, where $(\cdot)^*$ means the conjugate transpose of a matrix. \textcolor{rot}{For sufficiently large $|n|$, we have}
\ben
&&\widetilde{W}_n^{(1,1)}=\frac{2\mu(\lambda+2\mu)} {R(\lambda+3\mu)}|n|+O(1) >0,\\
&&\mbox{det}\,(\widetilde{W}_n) = \frac{4\mu^2(\lambda+2\mu)^2-[(\lambda-\widetilde{\lambda}) (\lambda+3\mu)+2\mu^2]^2}{R^2(\lambda+3\mu)^2}n^2+O(n).
\enn
Under the assumption (\ref{eq:7}) on $\widetilde{\lambda}$, we see
\ben
4\mu^2(\lambda+2\mu)^2-[(\lambda-\widetilde{\lambda}) (\lambda+3\mu)+2\mu^2]^2>0.
\enn
implying that $\mbox{det}\,(\widetilde{W}_n)>0$ for sufficiently large $|n|$. Hence, there exists $M>0$ such that $\widetilde{W}_n$ is positive definite over $\C^2$ for all $|n|\geq M$. This proves the positivity of the operator $-\real\widetilde{\mathcal{T}_1}$  defined in Lemma
\ref{Lem:generalizedDtN}. Finally, $\widetilde{\mathcal{T}}-\widetilde{\mathcal{T}_1}$ is compact since it is a finite dimensional operator over $\textcolor{rot}{(H^{1/2}(\Gamma_R))^2}$.
\end{proof}
Below we verify the analogue of Rellich's lemma in plane elasticity. It was used in the uniqueness proof of Theorem \ref{theorem}.
\begin{lemma}\label{lem:Rellich-2D}
Let $u^{sc}$ be a radiating solution to the Navier equation (\ref{eq:Naiver}) in $|x|\geq R$. Suppose that
\ben
\Ima\,\left(\int_{\Gamma_R} \widetilde{\mathcal{T}}(u^{sc}|_{\Gamma_R})\cdot \overline{u^{sc}}\,ds\right)=0.
\enn Then $u^{sc}\equiv 0$ in $|x|\geq R$.
\end{lemma}
\begin{proof}
Assume that $u^{sc}$ can be decomposed into the form of (\ref{radiation-solution}) and (\ref{radiation-solution-ps}) with the coefficients $\Psi_n=(\psi^n_p, \psi^n_s)^\top\in \C^2$. It follows from (\ref{series-DV}) and (\ref{DtN-matix}) that
\be\label{eq:5}
\int_{\Gamma_R} \widetilde{\mathcal{T}}(u^{sc}|_{\Gamma_R})\cdot \overline{u^{sc}}\,ds
=\sum_{n\in \Z}\left(R^{-2}B_n \Psi_n, R^{-1}A_n \Psi_n\right)
=R^{-3}\sum_{n\in \Z}\left(A_n^*B_n \Psi_n,\Psi_n\right).
\en
Using again the relations in (\ref{eq:4}), straightforward calculations show that
\be\label{eq:6}
A_n^*B_n=\begin{bmatrix}
t_p\overline{\gamma_p} & -in \\
-in & -t_s\overline{\gamma_s}
\end{bmatrix} \begin{bmatrix}
(\mu+\widetilde{\mu})t_p^2\beta_p -\widetilde{\lambda}t_p^2 & i(\mu+\widetilde{\mu})n\left( t_s\gamma_s-1\right) \\
i(\mu+\widetilde{\mu})n\left(t_p\gamma_p-1\right) & -\left(\mu+\widetilde{\mu}\right)t_s^2\beta_s -\widetilde{\mu}t_s^2
\end{bmatrix}=:\begin{bmatrix}
a_{11} & a_{12}\\
a_{21} & a_{22}
\end{bmatrix}.
\en
Recalling $
\widetilde{\lambda}+\widetilde{\mu}=\lambda+\mu
$ and making use of the relations
\ben
t_\alpha^2\,\beta_\alpha=n^2-t_\alpha^2-t_\alpha\gamma_\alpha, \quad
\Ima (\overline{\gamma}_\alpha)=-\frac{2}{|H_n^{(1)}(t_\alpha)|^2\pi t_\alpha}<0,\quad \alpha=p,s,
\enn we obtain
\ben
&&\Ima(a_{11})=-\Ima(\overline{\gamma}_p)(\lambda+2\mu)t_p^3=\frac{2\omega^2R^2}{\pi |H_n^{(1)}(k_pR)|^2 }>0,\\
&&\Ima(a_{22})=-\Ima(\overline{\gamma}_s)\mu t_s^3=\frac{2\omega^2R^2}{\pi |H_n^{(1)}(k_sR)|^2 }>0,\\
&&a_{12}=\overline{a}_{21}.
\enn
This implies that
\ben
\Ima (A_n^*B_n)=\frac{(A_n^*B_n)-(A_n^*B_n)^*}{2i}= \frac{2\omega^2R^2}{\pi}\begin{bmatrix}
1/|H_n^{(1)}(k_pR)|^2 & 0\\
0 & 1/|H_n^{(1)}(k_sR)|^2
\end{bmatrix}.
\enn
Now, we conclude from (\ref{eq:5}) and (\ref{eq:6}) that
\ben
0=\frac{2\omega^2}{\pi R}\sum_{n\in \Z}\left( \left|\frac{\psi^n_p}{H_n^{(1)}(k_pR)}\right|^2+
\left|\frac{\psi^n_s}{H_n^{(1)}(k_sR)}\right|^2  \right)
\enn
implying that $\psi^n_s=\psi^n_p=0$ for all $n\in \Z$. Therefore, $u^{sc}\equiv 0$ in $|x|\geq R$.
\end{proof}

\subsection{Properties of DtN \textcolor{rot}{map} in 3D}\label{prof-dtn-3d}
The aim of this section is to derive properties of the generalized DtN \textcolor{rot}{map} in 3D, following the lines in the previous section. Denote by $(r, \theta, \phi)$ the spherical coordinates of $x=(x_1, x_2, x_3)\textcolor{rot}{^\top}\in \R^3$.
The coordinate $\theta\in[0,\pi]$ corresponds to the angle from the $z$-axis,  whereas $\phi\in[0,2\pi)$ corresponds to the polar angle in the $(x,y)$-plane.
Let
   \ben
   \hat{\boldsymbol{r}}&=&(\cos\theta\sin\phi, \sin\theta\sin\phi, \cos\phi)\textcolor{rot}{^\top},\\ \hat{\boldsymbol{\theta}}&=&(-\sin\theta,\cos\theta,0)\textcolor{rot}{^\top}, \\
    \hat{\boldsymbol{\phi}}&=&(\cos\theta\cos\phi,\sin\theta\cos\phi,-\sin\phi) \textcolor{rot}{^\top}
   \enn
   be the unit vectors in the spherical coordinates.
In 3D, we need the $nm$-th spherical harmonic functions
\ben
Y_{nm}(\hat{\bb{x}}):=Y_{nm}(\theta,\phi)=\sqrt{\frac{(2n+1)(n-|m|)!} {4\pi(n+|m|)!}} P_n^{|m|}(\cos\theta)e^{im\phi},\quad \hat{\bb{x}}:=x/|x|\in \mathbb{S}^2
\enn for all $n\in \N$ and $m=-n,\cdots,n$,
where $P_n^m$ is the $m$-th associated Lagendre function of order $n$. Let $u_{nm}$ and $\bb{V}_{nm}$ be the vector spherical harmonics defined as
\be
\label{vsh2}
u_{nm}(\hat{\bb{x}}):= \frac{\nabla_{\mathbb{S}^2}Y_{nm}(\hat{\bb{x}})} {\sqrt{\delta_n}},\quad
\bb{V}_{nm}(\hat{\bb{x}}):= \hat{\bb{x}}\times u_{nm}(\hat{\bb{x}}),
\en
where $\delta_n:=n(n+1)$ and $\nabla_{\mathbb{S}^2}$ denotes the surface gradient on $\mathbb{S}^2$. They form a complete orthonormal basis in the $L^2$-tangent space of the unit sphere
\be
\label{tl2}
L_T^2(\mathbb{S}^2):=\left\{{\varphi}\in \textcolor{rot}{(L^2(\mathbb{S}^2))^3}: \hat{\bb{x}}\cdot{\varphi}(\hat{\bb{x}})=0\right\},
\en
 and satisfy the following  equations for any $f(r)\in C^1(\R^+)$:
\be
\label{relation1}
\mbox{curl}\,(f(r)\bb{V}_{nm}) &=& -\frac{\sqrt{\delta_n}f(r)}{r} Y_{nm}\hat{\bb{r}}-\frac{1}{r}\frac{\partial (rf(r))}{\partial r}u_{nm},\\
\label{relation2}
\hat{\bb{r}}\times\mbox{curl}\,(f(r)\bb{V}_{nm}) &=& -\frac{1}{r}\frac{\partial (rf(r))}{\partial r}\bb{V}_{nm},\\
\label{relation3}
\hat{\bb{r}}\times\mbox{curl}\,(f(r)Y_{nm}\hat{\bb{r}}) &=& \frac{\sqrt{\delta_n}f(r)}{r}u_{nm},\\
\label{relation4}
\mbox{div}\,(f(r)u_{nm}) &=& -\frac{\sqrt{\delta_n}f(r)}{r}Y_{nm}.
\en
As done in 2D,
we split a radiating solution $v^{sc}$ to the Navier equation (\ref{eq:Naiver}) into its compressional and  shear parts,
\be\label{eq:8}
v^{sc}=\mbox{grad}\,\psi_p+\bb{\psi}_s,\quad \mbox{div}\,\bb{\psi}_s=0.
\en
where $\psi_p$ is a scalar function satisfying
\be
\label{pre1}
\Delta\psi_p+k_p^2 =0,\quad
\lim_{r \to \infty} r\left(\frac{\partial \psi_p}{\partial r}-ik_p\psi_p\right) = 0,
\en
and the vector function $\bb{\psi}_s$ fulfills
\be
\label{she1}
\mbox{curl}\,\mbox{curl}\,\bb{\psi}_s-k_s^2\bb{\psi}_s = 0,\quad
\lim_{r \to \infty} r\left(\frac{\partial \bb{\psi}_s}{\partial r}-ik_s\bb{\psi}_s\right) =0.
\en
The solutions of (\ref{pre1}) and (\ref{she1}) in $|x|\geq R$ can be expressed as
\be\label{eq:9}
&&\psi_p=\sum_{n=0}^\infty\sum_{m=-n}^n \frac{h_n^{(1)}(k_pr)}{h_n^{(1)}(k_pR)}\psi_p^{nm} Y_{nm}(\theta,\phi),\\ \label{eq:10}
&&\bb{\psi}_s=\sum_{n=0}^\infty\sum_{m=-n}^n \left\{ \frac{h_n^{(1)}(k_sr)}{h_n^{(1)}(k_sR)}{\psi}_{s,1}^{nm} \bb{V}_{nm}(\theta,\phi) +\mbox{curl}\left[ \frac{h_n^{(1)}(k_sr)}{h_n^{(1)}(k_sR)}{\psi}_{s,2}^{nm} \bb{V}_{nm}(\theta,\phi)\right]\right\},
\en
where $\psi_p^{nm}, {\psi}_{s,j}^{nm} (j=1,2) \in \C$ and $h_n^{(1)}$ is the spherical bessel function of the third kind
of order $n$. A direct calculation implies that
\be
\label{generalsolution}
v^{sc}(x) &=& \sum_{n=0}^\infty\sum_{m=-n}^n \frac{h_n^{(1)}(k_sr)}{h_n^{(1)}(k_sR)} {\psi}_{s,1}^{nm}\bb{V}_{nm}(\theta,\phi)\no\\
&+& \sum_{n=0}^\infty\sum_{m=-n}^n \left\{ \frac{\sqrt{\delta_n}h_n^{(1)}(k_pr)}{rh_n^{(1)}(k_pR)}\psi_p^{nm} -\left[\frac{h_n^{(1)}(k_sr)}{rh_n^{(1)}(k_sR)} +\frac{k_s{h_n^{(1)}}'(k_sr)}{h_n^{(1)}(k_sR)}\right] {\psi}_{s,2}^{nm}\right\}u_{nm}(\theta,\phi)\no\\
&+& \sum_{n=0}^\infty\sum_{m=-n}^n \left\{ \frac{k_p{h_n^{(1)}}'(k_pr)}{h_n^{(1)}(k_pR)}\psi_p^{nm} -\frac{\sqrt{\delta_n}h_n^{(1)}(k_sr)}{rh_n^{(1)}(k_sR)} {\psi}_{s,2}^{nm}\right\}Y_{nm}(\theta,\phi)\hat{\bb{r}}.
\en
Analogously to the 2D case, we set
\be\label{eq:15}
t_\alpha:=k_\alpha R,\quad\gamma_\alpha:=\frac{{h_n^{(1)}}'(t_\alpha)}{h_n^{(1)}(t_\alpha)},\quad
\beta_\alpha:=\frac{{h_n^{(1)}}''(t_\alpha)}{h_n^{(1)}(t_\alpha)},\quad \alpha=p,s.
\en
Due to the orthogonality relations for $u_{nm}$, $\bb{V}_{nm}$ and $Y_{nm}\hat{\bb{r}}$ we derive  from
(\ref{generalsolution}) that
\ben
\left(v^{sc}|_{\G_R},\bb{V}_{nm}\right) &=& {\psi}_{s,1}^{nm},\\
\left(v^{sc}|_{\G_R},u_{nm}\right) &=& \frac{1}{R} \left[\sqrt{\delta_n}\psi_p^{nm}-(1+t_s\gamma_s){\psi}_{s,2}^{nm}\right],\\
\left(v^{sc}|_{\G_R},Y_{nm}\hat{\bb{r}}\right) &=& \frac{1}{R} \left(t_p\gamma_p\psi_p^{nm}-\sqrt{\delta_n}{\psi}_{s,2}^{nm}\right).
\enn
In other words,
\be
\label{SolR}
A_n\begin{bmatrix}
{\psi}_{s,1}^{nm} \\
{\psi}_{s,2}^{nm} \\
\psi_p^{nm}
\end{bmatrix}
=R\begin{bmatrix}
\left(v^{sc}|_{\G_R},\bb{V}_{nm}\right) \\
\left(v^{sc}|_{\G_R},u_{nm}\right) \\
\left(v^{sc}|_{\G_R},Y_{nm}\hat{\bb{r}}\right)
\end{bmatrix}
, \quad
A_n:=\begin{bmatrix}
R & 0 & 0 \\
0 & -1-t_s\gamma_s & \sqrt{\delta_n} \\
0 & -\sqrt{\delta_n} & t_p\gamma_p
\end{bmatrix}.
\en

\begin{lemma}
The matrix $A_n$ is invertible for all $n\geq 0$, $R>0$, $k_p>0$ and $k_s>0$. Its inverse is given by
\be
\label{MatrixA-3}
A_n^{-1}=\begin{bmatrix}
\frac{1}{R} & 0 & 0 \\
0 & \frac{t_p\gamma_p}{\Lambda_n} & -\frac{\sqrt{\delta_n}}{\Lambda_n}\\
0 & \frac{\sqrt{\delta_n}}{\Lambda_n} & \frac{-1-t_s\gamma_s}{\Lambda_n}
\end{bmatrix},\quad
\Lambda_n:=\delta_n-t_p\gamma_p(1+t_s\gamma_s).
\en
\end{lemma}
\begin{proof}
It's sufficient to prove that $\mbox{det}(A_n)\ne 0$, or equivalently, $\Lambda_n\ne 0$. Setting
$
I_n(z):=z{h_n^{(1)}}'(z)/h_n^{(1)}(z)$,
 we have
$\Lambda_n=\delta_n-I_n(t_p)-I_n(t_p)I_n(t_s)$. Recalling from \cite[Theorem 2.6.1]{Nedelec} that
\be\label{eq:13}
1\leq -\real I_n(z)\leq n+1,\quad 0<\Ima I_n(z)=\frac{1}{z |h_n^{(1)}(z)|^2}\leq z\quad \mbox{for all}\quad z>0,
\en we obtain
\ben
\Ima(\Lambda_n)=-\Ima I_n(t_p) (1+\real(I_n(t_s))-\real I_n(t_p)\Ima I_n(t_s)>0.
\enn

\end{proof}
The equation (\ref{SolR}) implies \textcolor{rot}{the following remark}.
\begin{remark}\label{Remark-3}
The unique radiating solution $v^{sc}$ to the boundary value problem (\ref{DtN}) can be represented in the form of (\ref{radiation-solution}) and (\ref{radiation-solution-ps}), where the coefficients $\psi_p^{n,m}$ and
    $\psi_{s,j}^{n,m}$ ($j=1,2$) are given by
    \ben
\begin{bmatrix}
\psi_{s,1}^{nm}\\
\psi_{s,2}^{nm}\\
\psi_p^{nm}
\end{bmatrix}
=R A_n^{-1}\begin{bmatrix}
\left(w,\bb{V}_{nm}\right) \\
\left(w,u_{nm}\right)\\
\left(w,Y_{nm}\hat{\bb{r}}\right)
\end{bmatrix}.
    \enn
\end{remark}
We now consider the generalized stress operator
\be\label{ge-stress}
\widetilde{\mathcal{T}}v^{sc}=(\mu+\widetilde{\mu})\hat{\bb{r}}\cdot\mbox{grad}\,v^{sc} + \widetilde{\lambda}\hat{\bb{r}}\,\mbox{div}\,v^{sc} + \widetilde{\mu}\hat{\bb{r}}\times\mbox{curl}\,v^{sc},
\en
where $\widetilde{\lambda}, \widetilde{\mu}\in \R$ satisfying $\widetilde{\lambda}+\widetilde{\mu}=\lambda+\mu$. Using the notation introduced in (\ref{eq:15}),
the first and second terms on the right hand side of (\ref{ge-stress}) can be rewritten respectively as
\ben
(\hat{\bb{r}}\cdot\mbox{grad}\,v^{sc})\big|_{\G_R} &=& \left(\frac{\partial v^{sc}}{\partial r}\right)\bigg|_{\G_R} \\
&=& \sum_{n=0}^\infty\sum_{m=-n}^n \frac{t_s\gamma_s}{R}{\psi}_{s,1}^{nm} \bb{V}_{nm}(\theta,\phi)\no\\
&+& \sum_{n=0}^\infty\sum_{m=-n}^n \frac{1}{R^2}\left[ \sqrt{\delta_n}\left(t_p\gamma_p-1\right)\psi_p^{nm} +\left(1-t_s\gamma_s-t_s^2\beta_s\right) {\psi}_{s,2}^{nm}\right]u_{nm}(\theta,\phi)\no\\
&+& \sum_{n=0}^\infty\sum_{m=-n}^n \frac{1}{R^2}\left[ t_p^2\beta_p\psi_p^{nm} +\sqrt{\delta_n}\left(1-t_s\gamma_s\right) {\psi}_{s,2}^{nm}\right]Y_{nm}(\theta,\phi)\hat{\bb{r}},
\enn
and
\ben
(\hat{\bb{r}}\,\mbox{div}\,v^{sc})\big|_{\G_R} &=& \left(\hat{\bb{r}}\Delta\psi_p\right)\big|_{\G_R}
= \left(-k_p^2\psi_p\hat{\bb{r}}\right)\big|_{\G_R}
= -\sum_{n=0}^\infty\sum_{m=-n}^n \frac{t_p^2}{R^2}\psi_p^{nm} Y_{nm}(\theta,\phi)\hat{\bb{r}}.
\enn
Since $h_n^{(1)}(k_sr)\bb{V}_{nm}(\theta,\phi)$ is a radiating solution of (\ref{she1}) and
\ben
\hat{\bb{r}}\times\bb{V}_{nm} = \hat{\bb{r}}\left(\hat{\bb{r}}\cdot u_{nm} \right)-u_{nm}(\hat{\bb{r}}\cdot\hat{\bb{r}})
= -u_{nm},
\enn
the third term of $\widetilde{\mathcal{T}}v^{sc}$ in (\ref{ge-stress}) takes the form
\ben
(\hat{\bb{r}}\times\mbox{curl}\,v^{sc})\big|_{\G_R} &=& (\hat{\bb{r}}\times\mbox{curl}\,\bb{\psi}_s)\big|_{\G_R} \\
&=& \sum_{n=0}^\infty\sum_{m=-n}^n \hat{\bb{r}}\times\mbox{curl}\left[ \frac{h_n^{(1)}(k_sr)}{h_n^{(1)}(k_sR)}{\psi}_{s,1}^{nm} \bb{V}_{nm}(\theta,\phi)\right] \\
&&+ \sum_{n=0}^\infty\sum_{m=-n}^n \left\{\hat{\bb{r}}\times\mbox{curl}\,\mbox{curl}\left[{\psi}_{s,2}^{nm} \frac{h_n^{(1)}(k_sr)}{h_n^{(1)}(k_sR)} \bb{V}_{nm}(\theta,\phi)\right]\right\} \\
&=& -\sum_{n=0}^\infty\sum_{m=-n}^n \frac{1}{R}\left(1+t_s\gamma_s\right){\psi}_{s,1}^{nm} \bb{V}_{nm}(\theta,\phi) \\
&&- \sum_{n=0}^\infty\sum_{m=-n}^n \frac{t_s^2}{R^2}{\psi}_{s,2}^{nm} u_{nm}(\theta,\phi).
\enn
Therefore,
\ben
\left(\widetilde{\mathcal{T}}v^{sc}|_{\G_R},\bb{V}_{nm}\right) &=& \frac{1}{R}\left(\mu t_s\gamma_s-\widetilde{\mu}\right){\psi}_{s,1}^{nm},\\
\left(\widetilde{\mathcal{T}}v^{sc}|_{\G_R},u_{nm}\right) &=& \frac{1}{R^2} \left\{\sqrt{\delta_n}(\mu+\widetilde{\mu})\left(t_p\gamma_p-1\right) \psi_p^{nm}+\left[(\mu+\widetilde{\mu})\left( 1-t_s\gamma_s-t_s^2\beta_s\right)-\widetilde{\mu}t_s^2\right] {\psi}_{s,2}^{nm}\right\},\\
\left(\widetilde{\mathcal{T}}v^{sc}|_{\G_R},Y_{nm}\hat{\bb{r}}\right) &=& \frac{1}{R^2} \left\{\left[(\mu+\widetilde{\mu})t_p^2\beta_p -\widetilde{\lambda}t_p^2 \right] \psi_p^{nm}+\sqrt{\delta_n}(\mu+\widetilde{\mu})\left(1-t_s\gamma_s\right) {\psi}_{s,2}^{nm}\right\}.
\enn
Set the matrices
\be
\label{MatrixB-3}
B_n:=\begin{bmatrix}
R\left(\mu t_s\gamma_s-\widetilde{\mu}\right) & 0 & 0 \\
0 & (\mu+\widetilde{\mu})\left( 1-t_s\gamma_s-t_s^2\beta_s\right)-\widetilde{\mu}t_s^2 & \sqrt{\delta_n}(\mu+\widetilde{\mu})\left(t_p\gamma_p-1\right) \\
0 & \sqrt{\delta_n}(\mu+\widetilde{\mu})\left(1-t_s\gamma_s\right) & (\mu+\widetilde{\mu})t_p^2\beta_p -\widetilde{\lambda}t_p^2
\end{bmatrix},
\en and define $W_n:=1/R B_nA_n^{-1}$.
Then we obtain
\be
\label{TractionR-3}
\begin{bmatrix}
\left(\widetilde{\mathcal{T}}v^{sc}|_{\G_R},\bb{V}_{nm}\right) \\
\left(\widetilde{\mathcal{T}}v^{sc}|_{\G_R},u_{nm}\right) \\
\left(\widetilde{\mathcal{T}}v^{sc}|_{\G_R},Y_{nm}\hat{\bb{r}}\right)
\end{bmatrix}
=B_n\begin{bmatrix}
{\psi}_{s,1}^{nm} \\
{\psi}_{s,2}^{nm} \\
\psi_p^{nm}
\end{bmatrix}=W_n \begin{bmatrix}
\left(v^{sc}|_{\G_R},\bb{V}_{nm}\right) \\
\left(v^{sc}|_{\G_R},u_{nm}\right) \\
\left(v^{sc}|_{\G_R},Y_{nm}\hat{\bb{r}}\right)
\end{bmatrix}.
\en
The above identity links the generalized stress operator $\widetilde{\mathcal{T}}v^{sc}|_{\G_R}$ and
$v^{sc}|_{\G_R}$ in the coordinate system $(\bb{V}_{nm}, u_{nm},Y_{nm}\hat{\bb{r}} )$ of the vector space
$(L^2(\s^2))^3$. Below we shall investigate properties of the three dimensional DtN \textcolor{rot}{map} $\widetilde{\mathcal{T}}$ using (\ref{TractionR-3}).
\begin{lemma}\label{Lem:generalizedDtN-3D}
\begin{description}
\item[(i)] $\widetilde{\mathcal{T}}$ is a bounded linear \textcolor{rot}{operator} from $\textcolor{rot}{(H^{s}(\Gamma_R))^3}$  to $\textcolor{rot}{(H^{s-1}(\Gamma_R))^3}$ for all $s\in \R$.
\item[(ii)] The matrix $-\real W_n$ is positive definite for sufficiently large $n>0$. Hence $\widetilde{\mathcal{T}}$ is the sum of a positive operator and a compact operator over  $\textcolor{rot}{(H^{1/2}(\Gamma_R))^3}$.
\item[(iii)] Lemma \ref{lem:Rellich-2D} remains valid for the generalized DtN \textcolor{rot}{map} $\widetilde{\mathcal{T}}$ in 3D.
\end{description}
\end{lemma}
\begin{proof} (i) We only need to show that the max norm of the matrix $W_n$ is bounded by
\be\label{maxnorm-3}
||W_n||_{\mbox{max}}=R^{-1}||B_n\, A_n^{-1}||_{\mbox{max}}\leq C\,n,
\en
for some constant $C>0$ uniformly in all $n>0$, where the matrices $A_n$ and $B_n$ are given by
(\ref{MatrixA-3}) and (\ref{MatrixB-3}), respectively. For this purpose we need to derive the asymptotics of each entry
$W_n^{(i,j)}$ ($1\leq i,j\leq 3$) of $W_n$. In three dimensions, it holds that
\ben
{h_n^{(1)}}''(z) &=& \left(h_{n-1}^{(1)}(z)-\frac{n+1}{z} h_n^{(1)}(z)\right)'\\
&=& -h_n^{(1)}(z)+\frac{n-1}{z}h_{n-1}^{(1)}(z) +\frac{n+1}{z^2}h_n^{(1)}(z)-\frac{n+1}{z}\left(h_{n-1}^{(1)}(z) -\frac{n+1}{z} h_n^{(1)}(z)\right)\\
&=& \frac{(n+1)^2+n+1-z^2}{z^2}h_n^{(1)}(z)-\frac{2}{z} h_{n-1}^{(1)}(z)\\
&=& \frac{(n+1)^2+n+1-z^2}{z^2}h_n^{(1)}(z)-\frac{2}{z} \left({h_n^{(1)}}'(z)+\frac{n+1}{z} h_n^{(1)}(z)\right)\\
&=& \left(\frac{\delta_n}{z^2}-1\right)h_n^{(1)}(z)- \frac{2}{z} {h_n^{(1)}}'(z),
\enn
implying that
\be\label{eq:11}
\beta_p=\frac{\delta_n}{t_p^2}-1-\frac{2}{t_p}\gamma_p, \quad \beta_s=\frac{\delta_n}{t_s^2}-1-\frac{2}{t_s}\gamma_s.
\en
Note that the relations in (\ref{eq:11}) differ from those in two dimensions; cf. (\ref{eq:4}).
Using the expressions of $B_n$ and $A_n$, we obtain the entries of $W_n$ via straightforward calculations
\ben
W_n^{(1,1)} &=& \frac{\mu t_s\gamma_s-\widetilde{\mu}}{R},\\
W_n^{(1,2)}&=& W_n^{(2,1)}=W_n^{(1,3)}=W_n^{(3,1)}=0,\\
W_n^{(2,2)} &=& \frac{1}{R\Lambda_n}\left[ t_p\gamma_p(\mu+\widetilde{\mu})\left( 1-t_s\gamma_s-t_s^2\beta_s\right)-\widetilde{\mu}t_s^2t_p\gamma_p -\delta_n(\mu+\widetilde{\mu})\left(1-t_p\gamma_p\right)\right]\\
&=& \frac{1}{R\Lambda_n}\left[ (\mu+\widetilde{\mu})(t_pt_s\gamma_p\gamma_s-\delta_n+t_p\gamma_p) +\mu t_s^2t_p\gamma_p\right], \\
W_n^{(3,3)} &=& \frac{1}{R\Lambda_n}\left\{ -\delta_n(\mu+\widetilde{\mu})\left(1-t_s\gamma_s\right) -\left[(\mu+\widetilde{\mu})t_p^2\beta_p -\widetilde{\lambda}t_p^2\right]\left(1+t_s\gamma_s\right)\right\}\\
&=& \frac{1}{R\Lambda_n}\left[ t_p^2(\lambda+2\mu)(1+t_s\gamma_s) +2(\mu+\widetilde{\mu}) (t_pt_s\gamma_p\gamma_s-\delta_n+t_p\gamma_p)\right],\\
W_n^{(2,3)} &=& \frac{1}{R\Lambda_n}\left[ -\sqrt{\delta_n}(\mu+\widetilde{\mu})\left( 1-t_s\gamma_s-t_s^2\beta_s\right)+\widetilde{\mu}t_s^2\sqrt{\delta_n} +\sqrt{\delta_n}(\mu+\widetilde{\mu})\left(1-t_p\gamma_p\right) \left(1+t_s\gamma_s\right)\right] \\
&=& \frac{1}{R\Lambda_n}\left[ \sqrt{\delta_n}(\mu+\widetilde{\mu})(\delta_n-t_pt_s\gamma_p\gamma_s -t_p\gamma_p)-\mu t_s^2\sqrt{\delta_n}\right],\\
W_n^{(3,2)} &=& \frac{1}{R\Lambda_n}\left[ \sqrt{\delta_n}t_p\gamma_p(\mu+\widetilde{\mu}) \left(1-t_s\gamma_s\right) +t_p^2\beta_p\sqrt{\delta_n}(\mu+\widetilde{\mu}) -\widetilde{\lambda}t_p^2\sqrt{\delta_n}\right] \\
&=& \frac{1}{R\Lambda_n}\left[ \sqrt{\delta_n}(\mu+\widetilde{\mu})(\delta_n-t_pt_s\gamma_p\gamma_s -t_p\gamma_p)-(\lambda+2\mu)t_p^2\sqrt{\delta_n}\right],
\enn
in which we have used the relation (\ref{eq:11}) and the fact that
$\widetilde{\lambda}+\widetilde{\mu}=\lambda+\mu.$
Now, we need to derive the asymptotics of $W_n^{(i,j)}$ ($1\leq i,j \leq 3$) as $|n|$ tends to infinity.
From the series expansions of the spherical Bessel and Neumann functions we know
\ben
h_n^{(1)}(z)=\frac{1}{i}1\cdot 3\cdot\cdots\cdot (2n-1)\left[\frac{1}{z^{n+1}} +\frac{1}{2z^{n-1}(2n-1)}+O\left(\frac{1}{n^2}\right)\right],\quad n\rightarrow +\infty.
\enn
Then
\ben
\frac{h_{n-1}^{(1)}(z)}{h_n^{(1)}(z)} &=& \frac{1}{2n-1} \frac{\frac{1}{z^n} +\frac{1}{2z^{n-2}(2n-3)}+O\left(\frac{1}{n^2}\right)} {\frac{1}{z^{n+1}} +\frac{1}{2z^{n-1}(2n-1)}+O\left(\frac{1}{n^2}\right)}\\
&=& \left[\frac{1}{2n}+O\left(\frac{1}{n^2}\right)\right] \left[z+O\left(\frac{1}{n}\right)\right] \\
&=& \frac{z}{2n} +O\left(\frac{1}{n^2}\right),
\enn
which further leads to
\ben
\frac{{h_n^{(1)}}'(z)}{h_n^{(1)}(z)}=\frac{z}{2n}  -\frac{n+1}{z}+O\left(\frac{1}{n^2}\right),\quad n\rightarrow +\infty.
\enn
Therefore, as $n\rightarrow +\infty$,
\ben
W_n^{(1,1)} &=& -\frac{\mu}{R}n-\frac{\mu+\widetilde{\mu}}{R} +O\left(\frac{1}{n}\right),\\
W_n^{(2,2)} &=& -\frac{2\mu(\lambda+2\mu)} {R(\lambda+3\mu)}n+O(1),\\
W_n^{(3,3)} &=& -\frac{2\mu(\lambda+2\mu)} {R(\lambda+3\mu)}n+O(1),\\
W_n^{(2,3)} &=& \frac{\left[(\mu+\widetilde{\mu}) (\lambda+3\mu)-2\mu(\lambda+2\mu) \right]}{R(\lambda+3\mu)}n +O(1),\\
W_n^{(3,2)} &=& \frac{\left[(\mu+\widetilde{\mu}) (\lambda+3\mu)-2\mu(\lambda+2\mu) \right]}{R(\lambda+3\mu)}n +O(1).
\enn
This proves (\ref{maxnorm-3}) and thus the first assertion.

(ii) Set $\widetilde{W}_n:=-(W_n+W_n^*)/2$ for $n\ge 0$.
For sufficiently large $n>0$, we have
\ben
&&\widetilde{W}_n^{(1,1)}=\frac{\mu}{R}n+\frac{\mu+\widetilde{\mu}}{R} +O\left(\frac{1}{n}\right) >0,\\
&&\widetilde{W}_n^{(1,1)}\widetilde{W}_n^{(2,2)}= \frac{2\mu^2(\lambda+2\mu)} {R^2(\lambda+3\mu)}n^2+O(n)>0,\\
&&\mbox{det}\,(\widetilde{W}_n) =\widetilde{W}_n^{(1,1)} \left(\frac{4\mu^2(\lambda+2\mu)^2-[(\lambda-\widetilde{\lambda}) (\lambda+3\mu)+2\mu^2]^2}{R^2(\lambda+3\mu)^2}n^2+O(n)\right).
\enn
Recalling the assumption (\ref{eq:7}) on $\widetilde{\lambda}$ we see
\ben
4\mu^2(\lambda+2\mu)^2-[(\lambda-\widetilde{\lambda}) (\lambda+3\mu)+2\mu^2]^2>0.
\enn
This implies that $\mbox{det}\,\widetilde{W}_n$ is positive definite over $\C^3$ for sufficiently large $n$. The proof of the second assertion is compete.

(iii) Assume that a radiating solution $\textcolor{rot}{v^{sc}}$ to the Navier equation (\ref{eq:Naiver}) admits the series expansion (\ref{eq:8}), (\ref{eq:9}) and (\ref{eq:10})  with the vector coefficient $\Psi^{nm}:= ({\psi}_{s,1}^{nm}, {\psi}_{s,2}^{nm}, \psi_p^{nm} )^\top\in \C^3$.
Making use of (\ref{SolR}) and the first relation in (\ref{TractionR-3}), we get
\ben
\int_{\Gamma_R} \widetilde{\mathcal{T}}(v^{sc}|_{\Gamma_R})\cdot \overline{v^{sc}}\,ds
&=&\sum_{n\in \N_0}\sum_{m=-n}^n\left<R^{-2}B_n ^{nm}, R^{-1}A_n \Psi^{nm}\right>\\
&=&R^{-3}\sum_{n\in \N_0}\sum_{m=-n}^n\left<A_n^*B_n \Psi^{nm},\Psi^{nm}\right>.
\enn Here $\left<\cdot,\cdot\right>$ denotes the inner product over $\C^3$. Hence,
\be \label{eq:12}
\sum_{n\in \N_0}\sum_{m=-n}^n\left<\Ima (A_n^*B_n) \Psi^{nm},\Psi^{nm}\right>=0.
\en
To evaluate the product of $A_n^*$ and $B_n$ we need the identities (cf. (\ref{eq:13}), (\ref{eq:11}))
\be\label{eq:14}
\Ima (t_\alpha \gamma_\alpha)=1/(t_\alpha |h_n^{(1)}(t_\alpha)|^2)>0,
\qquad t_\alpha^2 \beta_\alpha=\delta_n-t_{\alpha}^2-2t_\alpha\gamma_\alpha,\qquad\alpha=p,s.
\en
Since
\ben
A_n^*=\begin{bmatrix}
R & 0 & 0 \\
0 & -1-t_s\overline{\gamma}_s & -\sqrt{\delta_n} \\
0 & \sqrt{\delta_n} & t_p\overline{\gamma}_p
\end{bmatrix},
\enn
direct calculations show that
\ben
\Ima (A_n^*B_n)=R^2\,\begin{bmatrix}
\mu \Ima(t_s\gamma_s) & 0 & 0\\
0    & \omega^2 \Ima(t_s\gamma_s) & 0\\
0    & 0 & \omega^2 \Ima(t_s\gamma_s)
\end{bmatrix}.
\enn
This together with (\ref{eq:12}) and the first relation in (\ref{eq:14})  yields $|\Psi^{nm}|=0$ for all $n\geq 0$, $m=-n,\cdots, n$. Therefore, $v^{sc}\equiv 0$ in $|x|\geq R$.
\end{proof}

\section{Reconstruction of multiple anisotropic obstacles}\label{sec:inverse}

In this section, we consider the inverse scattering problem of reconstructing the support of multiple unknown anisotropic obstacles from near-field measurement data. We first derive the Fr\'echet derivative of the near-field solution operator, which maps the boundaries of several disconnected scatterers to the measurement data. Then, as an application, we design an iterative approach to the inverse problem using the data of one or several incident directions and frequencies.

\subsection{Fr\'echet derivative of the solution operator}

Suppose that $\Om=\cup_{i=1}^{N_0} \Om_j$ is a union of several disconnected bounded components $\Omega_j\subset\R^N$.
Each component $\Om_j$ is supposed to be occupied by an anisotropic elastic obstacle with constant density $\rho_j>0$ and constant stiffness tensor $\mathcal{C}_j=(C_{j,klmn})_{k,l,m,n=1}^N$.
Assume that the boundary $\G_j$ of $\Om_j$ is $C^2$. Let $\Om_0:=B_R\backslash\ov{\Om}$.
Denote by $\rho_0>0$ and $\mathcal{C}_0=(C_{0,klmn})_{k,l,m,n=1}^N$ the density and stiffness tensor of the homogeneous isotropic background medium.
Set
\be\label{u}
u:=\begin{cases}
u_j, & x\in\Omega_j, \cr
u^{sc}+u^{in}, & x\in\R^N\backslash\ov{\Omega}.
\end{cases}
\en
We assume there is an a priori information that the unknown elastic scatterers $\Om_j$, $j=1,\cdots,N_0$, are embedded in the region $B_R$ for some $R>0$.
The variational formulation for the forward scattering problem in the truncated domain $B_R$ reads as follows: find $u\in X_R:=(H^1(B_R))^N$ such that
\be
\label{variational1}
a(u,v) =
\int_{\Gamma_R} f\cdot \ov{v}\,ds \quad\mbox{for all}\quad v\in X_R,\quad f:=(Tu^{in}-\mathcal{T}u^{in})|_{\Gamma_R},
\en
where
\ben
a(u,v)&:=&\sum_{j=0}^{N_0} A_{\Omega_j}(u,v)-\int_{\Gamma_R} \mathcal{T}u\cdot\ov{v}\,ds\\
A_{\Om_j}(u,v)&:=&\int_{\Omega_j} \left(\sum_{k,l,m,n=1}^N C_{j,klmn}\frac{\partial u_m}{\partial x_n}\frac{\partial \ov{v_k}}{\partial x_l} -\rho_j\,\omega^2u\cdot\ov{v}\right)dx,\quad j=0,1,\cdots, N_0.
\enn
Here  $\mathcal{T}$ is the DtN map introduced in the previous section.
  We study the following inverse problem:
\begin{description}
\item[\bf (IP):] Determine the boundaries $\G_1,\cdots,\G_{N_0}$ from knowledge of multi-frequency near-field measurements $u|_{\G_R}$ corresponding to the incident plane wave (\ref{planewave}) with one or several incident directions.
\end{description}

Let $u\in X_R$ be the unique solution to the variational problem (\ref{variational1}). Since each boundary $\G_j$ is $C^2$, we have $u\in (H^2(B_R))^N$. In this paper we define the near-field solution operator $J$ as
\be
\label{solutionOpera}
\mathcal{J}:\quad (\G_1,\cdots,\G_{N_0})\rightarrow u|_{\Gamma_R}.
\en
The mapping $J$ is obviously nonlinear.
To define the Fr\'echet derivative of $\mathcal{J}$ with respect to the boundary $\Gamma=\cup_{j=1}^{N_0} \Gamma_j$, we assume that the function
 \ben
 h_j=(h_{j,1},\cdots,h_{j,N})^\top\in (C^1(\G_j))^N,\quad \|h_j\|_{(C^1(\G_j))^N}\ll 1
 \enn
 is a small perturbation of $\G_j$. The perturbed boundary is given by
 \ben
 \G_{j,h}:=\{y\in\R^N: y=x+h_j(x), x\in\G_j\}.
 \enn
 \begin{definition}\label{def}
 The solution operator $\mathcal{J}$ is called Fr\'echet differentiable at $\G$ if there exists a linear bounded operator $\mathcal{J}_\G':(C^1(\G_1))^N\times\cdots\times (C^1(\G_{N_0}))^N \rightarrow (L^2(\G_R))^N$ such that
\ben
\|\mathcal{J}(\G_{1,h},\cdots,\G_{N_0,h}) -\mathcal{J}(\G_1,\cdots,\G_{N_0}) -\mathcal{J}_\G'(h_1,\cdots,h_{N_0})\|_{(L^2(\G_R))^N} =o\left(\sum_{j=1}^{N_0}\|h_j\|_{(C^1(\G_j))^N}\right).
\enn
The operator $\mathcal{J}_\G'$ is called the Fr\'echet derivative of $\mathcal{J}$ at $\G$.
\end{definition}
Given $h_j\in (C^1(\G_j))^N$, there exists an extension of $h_j$,  which we still denote by $h_j$, such that $h_j\in (C^1(\R^N))^N$, $\|h_j\|_{(C^1(\R^N))^N}\le c\|h_j\|_{(C^1(\G_j))^N}$ and $\mbox{supp}\,(h_j)\subset K_j$, where $K_j$ is a domain satisfying $\G_j\subset K_j\subset\subset B_R\backslash\ov{\left(\cup_{i=1,i\ne j}^{N_0} \Om_j\right)}$. Define the functions
\ben
h(x):=\sum_{j=1}^{N_0} h_j(x),\quad y=\xi^h(x)=x+h(x), \qquad x\in\R^N.
\enn
For small perturbations, $\xi^h$ is a diffeomorphism between $\G_j$ and $\G_{j,h}$. The inverse map of $\xi^h$ is denoted by $\eta^h$.
Corresponding to $\Omega_j$ ($j=0,1,\cdots,N_0$), we define
 \ben
 \Omega_{j,h}:=\{y\in\R^N: y=\xi^h(x), x\in\Om_j\},\quad j=1,2,\cdots,N_0,\qquad
 \Om_{0,h}:=B_R\backslash\ov{\cup_{j=1}^{N_0}\Omega_{j,h}}.
\enn
The differentiability of $\mathcal{J}$ at $\Gamma$ is stated as following.
\begin{theorem}
\label{theorem3.1}
Let $u$ (see (\ref{u})) be the unique solution of the variational problem (\ref{variational1}), and
let $h_j\in (C^1(\G_j))^N$, $j=1,\cdots,{N_0}$, be sufficiently small perturbations. Then the solution operator $\mathcal{J}$ is Fr\'echet differentiable at $\G$. Further, the Fr\'echet derivative $\mathcal{J}_\G'$ is given by $\mathcal{J}_\G'(h_1,\cdots,h_{N_0})=\widetilde{u}_0|_{\G_R}$, where $\widetilde{u}_0$ together with $\widetilde{u}_j$ ($j=1,\cdots,N_0$) is the unique weak solution of the boundary value problem:
\be
\label{derivative1}
\nabla\cdot (\mathcal{C}_j: \nabla \widetilde{u}_j)+\rho_j\omega^2 \widetilde{u}_j = 0&&\mbox{in}\quad\Omega_j,\; j=0,1,\cdots,{N_0},\\
\label{derivative3}
\widetilde{u}_j-\widetilde{u}_0 -f_j=0 &&\mbox{on}\quad\Gamma_j,\; j=1,\cdots,{N_0},\\
\label{derivative4}
\mathcal{N}_{\mathcal{C}}^-\widetilde{u}_j -\mathcal{N}_{\mathcal{C}}^+\widetilde{u}_0 -g_j=0 &&\mbox{on}\quad\Gamma_j,\; j=1,\cdots,{N_0},\\
\label{derivative5}
T \widetilde{u}_0 - \mathcal{T}\widetilde{u}_0=0&&\mbox{on}\quad\Gamma_R.
\en
where
\be
\label{dataf}
f_j = -(h_j\cdot\nu)\left[{\partial_\nu^- u_j}-{\partial_\nu^+ (u^{sc}+u^{in})}\right]|_{\Gamma_j}
\en
and the expressions of $g_j\in (H^{-1/2}(\Gamma_j))^N$ rely on the space dimensions. In 2D, we have
\be
\label{datagN2}
g_j &=& \omega^2(h_j\cdot\nu)\left[\rho_j u_j^--\rho_0 (u^{sc}+u^{in})^+\right]\nonumber\\
&&-\partial_\tau\left[\left((\sigma_j(u_j))^- -(\sigma_0(u^{sc}+u^{in}))^+\right)(h_{j,2},-h_{j,1})^\top\right],
\en
where $\partial_\tau=\nu^\perp\cdot\nabla$ is the tangential derivative. In 3D, it holds that
\be
\label{datagN3}
g_j = \omega^2(h_j\cdot\nu)\left[\rho_j u_j^--\rho_0 (u^{sc}+u^{in})^+\right]-\mbox{div}_{\G_j} \left((\bb{A}_j-\bb{A}_0)\times\nu\right),
\en
where $\mbox{div}_\G$ is the surface divergence operator on $\Gamma$ and $\bb{A}_j\in \C^{N\times N}$ are defined by
\be\label{AJ}
\bb{A}_j=\sigma_j(u_j)^-|_{\Gamma_j}\;\begin{bmatrix}
0 & -h_3 & h_2\\
-h_3 & 0 & h_1\\
h_2 & h_1 & 0
\end{bmatrix},\quad j=0,1,\cdots, N_0.
\en
\end{theorem}
\begin{proof}
Set the space
\ben
\mathcal{H}:=\{(v,w)\in (H^1(\Omega))^N\times (H^1(\Om_0))^N : v=w\quad\mbox{on}\quad\Gamma_j,\quad j=1,\cdots,{N_0} \}.
\enn
The variational problem of (\ref{derivative1})-(\ref{derivative5}) can be formulated as the problem of finding $\widetilde{u}_0\in (H^1(\Om_0))^N$, $\widetilde{u}_j\in (H^1(\Omega_j))^N$ such that $\widetilde{u}_j-\widetilde{u}_0=f_j$ on $\Gamma_j$, $j=1,\cdots,{N_0}$, and
\be
\label{variational2}
\sum_{j=0}^{N_0} A_{\Omega_j}(\widetilde{u}_j,v)-\int_{\Gamma_R} \mathcal{T}\widetilde{u}_0\cdot\ov{w}\,ds =
\sum_{j=1}^{N_0} \int_{\Gamma_j} g_j\cdot \ov{w}\,ds \quad\mbox{for all}\quad (v,w)\in\mathcal{H}.
\en
It follows from the regularity of $u$ that $f_j\in (H^{1/2}(\G_j))^N$ and $g_j\in (H^{-1/2}(\G_j))^N$. Let $\hat{f}_j\in (H^1(\Omega_j))^N$ be the trace lifting functions of $f_j$. Then the variational formulation (\ref{variational2}) \textcolor{rot}{search} for $\widetilde{u}_0\in (H^1(\Om_0))^N$ and $\hat{u}_j=\widetilde{u}_j-\hat{f}_j\in (H^1(\Omega_j))^N$ such that $\hat{u}_j=\widetilde{u}_0$ on $\Gamma_j$, $j=1,\cdots,{N_0}$ and
\be
\label{variational3}
\sum_{j=0}^{N_0} A_{\Omega_j}(\hat{u}_j,v)-\int_{\Gamma_R} \mathcal{T}\widetilde{u}_0\cdot\ov{w}\,ds  =
\sum_{j=1}^{N_0} \int_{\Gamma_j} g_j\cdot \ov{w}\,ds-\sum_{j=1}^{N_0} A_{\Omega_j}(\hat{f}_j,v) \quad\mbox{for all}\quad (v,w)\in\mathcal{H}.
\en
Applying Lemma \ref{DtN-lemma} and Theorem \ref{strong-elliptic}, we see that the above variational equation (\ref{variational3}) admits a unique solution. For the given functions $h_j\in (C^1(\Gamma_j))^N$, we extend them to $B_R$ in the same way as \textcolor{rot}{before}. Let $J_{\eta^h}$ and $J_{\xi^h}$ be the Jacobian matrices of the transforms $\eta^h$ and $\xi^h$, respectively. It then follows that
\ben
J_{\xi^h} &=& I+\nabla h,\quad
J_{\eta^h} = I-\nabla h+O\left(\|h\|_{C^1(B_R)^N}^2\right),\\
\mbox{det}(J_{\xi^h}) &=& 1+\nabla\cdot h+O\left(\|h\|_{C^1(B_R)^N}^2\right).
\enn
Consider the perturbed variational problem: find $u_h\in X_R$ such that
\be
\label{perturbedvariational}
\sum_{j=0}^{N_0} A_{\Omega_{j,h}}(u_h,v_h)-\int_{\Gamma_R} \mathcal{T}u_h\cdot\ov{v_h}\,ds =
\int_{\Gamma_R} f\cdot \ov{v_h}\,ds \quad\mbox{for all}\quad v_h\in X_R.
\en Here $f=(Tu^{in}-\mathcal{T}u^{in})|_{\Gamma_R}$.
Define $\widehat{u}=(\widehat{u}_1,\widehat{u}_2)^\top:=(u_h\circ\xi^h)(x)$. Then we have
\ben
\sum_{j=0}^{N_0} A_{\Omega_{j,h}}(u_h,v_h)
&=& \sum_{j=0}^{N_0}\int_{\Omega_j} \sum_{k,l,m,n=1}^N C_{j,klmn}\nabla^\top \widehat{u}_m J_{\eta^h}(:,n)J_{\eta^h}(:,l)^\top\nabla\ov{\widehat{v}_k}\,\mbox{det}(J_{\xi^h})\,dx\\
&&- \sum_{j=0}^{N_0}\rho_j\omega^2\int_{\Omega_j} \widehat{u}\cdot\ov{\widehat{v}}\,\mbox{det}(J_{\xi^h})\,dx
\enn
where $A(:,n)$ means the $n$-th column of the matrix $A$. From the stability of the direct scattering problem it follows that $\widehat{u}$ converges to $u$ in $X_R$ as $\|h\|_{(C^1(B_R))^N}\rightarrow 0$. Let $w\in X_R$ be the solution of the variational problem
\ben
a(w,v)= \sum_{j=0}^{N_0} b_j(u,v,h) \quad\mbox{for all}\quad v\in X_R,
\enn
where
\be\nonumber
b_j(u,v,h)&:=& \int_{\Omega_j}\sum_{k,l,m,n=1}^NC_{j,klmn}\left[\frac{\partial u_m}{\partial x_n}\frac{\partial h^\top}{\partial x_l}\nabla\ov{v_k} +\nabla^\top u_m\frac{\partial h}{\partial x_n}\frac{\partial \ov{v_k}}{\partial x_l} -(\nabla\cdot h)\frac{\partial u_m}{\partial x_n}\frac{\partial \ov{v_k}}{\partial x_l}\right]dx\nonumber\\ \label{b_j}
&&+ \rho_j\,\omega^2\int_{\Omega_j} (\nabla\cdot h)u\cdot\ov{v}\,dx.
\en
Then it's easy to prove that
\ben
\sup_{v\in X_R}\,a(\widehat{u}-u-w,v)/\|v\|_{X_R} =o\left(\|h\|_{(C^1(B_R))^N}\right).
\enn
Applying the trace theorem it follows that $a(\widehat{u}-u-w)/\|h\|_{(C^1(B_R))^N}$ tends to zero in $(H^{1/2}(\G_R))^N$ as $\|h\|_{(C^1(B_R))^N}$ tends to zero. By \textcolor{rot}{Definition} \ref{def}, we get $\mathcal{J}_\G'(h_1,\cdots,h_{N_0})=w|_{\G_R}$. Hence, it only remains to prove that
$w=\widetilde{u}_0$ on $\G_R$.

Below we are going to calculate $b_j(u,v,h)$ for $j=0,1,\cdots, N_0$.
 Set $(v,w)\in\mathcal{H}$. Using integration by parts and the relation
\ben
(\nabla\cdot h)(u\cdot v)=\nabla\cdot\left[(u\cdot v)h\right]-(h\cdot\nabla u)\cdot v-(h\cdot\nabla v)\cdot u,
\enn
the last term of (\ref{b_j}) can be \textcolor{rot}{written} as
\be\label{eq:16}
\rho_j\omega^2\int_{\Omega_j} (\nabla\cdot h)u\cdot\ov{v}\,dx
= \rho_j\omega^2 \int_{\Gamma_j} (h\cdot\nu)(u\cdot\ov{v})ds
- \rho_j\omega^2 \int_{\Omega_j} \left[(h\cdot\nabla u_j)\cdot\ov{v}+(h\cdot\nabla \ov{v})\cdot u\right]dx.
\en To compute the first integral on the right hand of (\ref{b_j}), we need the identities
\ben
\frac{\partial u_m}{\partial x_n}\frac{\partial h^\top}{\partial x_l}\nabla\ov{v_k}
&=& \frac{\partial}{\partial x_l}\left(\frac{\partial u_m}{\partial x_n}h^\top\nabla\ov{v_k}\right) -\frac{\partial^2u_m}{\partial x_l\partial x_n}\left(h^\top\nabla\ov{v_k}\right) -\frac{\partial u_m}{\partial x_n}h^\top\nabla\left(\frac{\partial \ov{v_k}}{\partial x_l}\right),\\
\nabla^\top u_m\frac{\partial h}{\partial x_n}\frac{\partial \ov{v_k}}{\partial x_l}
&=& \frac{\partial}{\partial x_n}\left(h^\top\nabla u_m\right)\frac{\partial \ov{v_k}}{\partial x_l}-\nabla\cdot\left(h\frac{\partial u_m}{\partial x_n}\frac{\partial \ov{v_k}}{\partial x_l}\right) +(\nabla\cdot h)\frac{\partial u_m}{\partial x_n}\frac{\partial \ov{v_k}}{\partial x_l}+ \frac{\partial u_m}{\partial x_n}h^\top\nabla\left(\frac{\partial \ov{v_k}}{\partial x_l}\right).
\enn
Making use of the previous two identities and applying again the integration by parts, it follows for $j\geq 1$ that
\be\no
&\quad& \int_{\Omega_j}\sum_{k,l,m,n=1}^NC_{j,klmn} \left[\frac{\partial u_m}{\partial x_n}\frac{\partial h^\top}{\partial x_l}\nabla\ov{v_k} +\nabla^\top u_m\frac{\partial h}{\partial x_n}\frac{\partial \ov{v_k}}{\partial x_l} -(\nabla\cdot h)\frac{\partial u_m}{\partial x_n}\frac{\partial \ov{v_k}}{\partial x_l}\right]dx\\ \no
&=& \int_{\Omega_j} \nabla\cdot \left[(h\cdot\nabla\ov{v})\cdot\sigma_j(u)-h(\sigma_j(u):\nabla\ov{v})\right]\,dx + \rho_j\omega^2 \int_{\Omega_j} (h\cdot\nabla \ov{v})\cdot u\,dx\\ \no
&&+ \int_{\Omega_j} \sum_{k,l,m,n=1}^N C_{j,klmn}\frac{\partial}{\partial x_n}(h\cdot\nabla u_m)\frac{\partial \ov{v_k}}{\partial x_l}\,dx\\ \no
&=& \int_{\Gamma_j} \left[(h_j\cdot\nabla\ov{v})\cdot (\nu\cdot\sigma_j(u_j)) -(h_j\cdot\nu)(\sigma_j(u_j):\nabla\ov{v})\right]\,ds + \rho_j\omega^2 \int_{\Omega_j} (h\cdot\nabla \ov{v})\cdot u\,dx\\ \label{eq:17}
&&+ \int_{\Omega_j} \sum_{k,l,m,n=1}^N C_{j,klmn}\frac{\partial}{\partial x_n}(h\cdot\nabla u_m)\frac{\partial \ov{v_k}}{\partial x_l}\,dx.
\en Next, we proceed with the space dimensions.
In two dimensions (i.e., $N=2$), we have
\ben
(h_j\cdot\nabla\ov{v})\cdot(\nu\cdot\sigma_j(u_j)) -(h_j\cdot\nu)(\sigma_j(u_j):\nabla\ov{v}) = \sigma_j(u_j)(h_{j,2},-h_{j,1})^\top\cdot\partial_\tau\ov{v}\quad\mbox{on}\;\Gamma_j.
\enn
Therefore, combining (\ref{b_j}), (\ref{eq:16}) and (\ref{eq:17}) yields
\ben
b_j(u,v,h) = A_{\Omega_j}(h\cdot\nabla u,v)+\rho_j\omega^2 \int_{\Gamma_j} (h_j\cdot\nu)(u_j\cdot\ov{v})ds -\int_{\Gamma_j} \partial_\tau\left[(\sigma_j(u_j))^-(h_{j,2},-h_{j,1})^\top\right]\cdot\ov{v}ds,
\enn for $j\geq 1$.
When $j=0$, we obtain in a similar manner that
\ben
b_0(u,v,h) &=& A_{\Omega_0}(h\cdot\nabla u,v)\\
&-&\sum_{j=1}^{N_0}\left\{\rho_0\omega^2 \int_{\Gamma_j} (h_j\cdot\nu)[(u^{sc}+u^{in})\cdot\ov{w}]ds -\int_{\Gamma_j} \partial_\tau\left[(\sigma_0(u^{sc}+u^{in}))^+(h_{j,2},-h_{j,1})^\top\right]\cdot\ov{w}ds \right\}.
\enn
Now define $\widetilde{u}=w-h\cdot\nabla u$ and set $\widetilde{u}_j:=\widetilde{u}|_{\Omega_j}$ for $j=0,1,\cdots,N_0$. We conclude that $\widetilde{u}_0|_{\Gamma_R}=w|_{\Gamma_R}$ and the formula (\ref{variational2}) holds with such $\widetilde{u}$. Furthermore, we have the transmission conditions
\ben
\widetilde{u}_j-\widetilde{u}_0 = -h_j\cdot[\nabla u_j-\nabla(u^{sc}+u^{in})]
= -(h_j\cdot\nu)[\partial_\nu^- u_j-\partial_\nu^+ (u^{sc}+u^{in})] \quad\mbox{on}\quad\Gamma_j,
\enn
since $u_j-u^{sc}=u^{in}$ on $\Gamma_j$. This prove the relation $\mathcal{J}_\G'(h_1,\cdots,h_{N_0})=\widetilde{u}_0|_{\G_R}$ in two dimensions.

If $N=3$, we recall the tangential gradient $\nabla_{\G}$ for a scalar function $u$ and the surface divergence $\mbox{div}_\G$ for a vector function $v$ by
\be
\label{notation}
\nabla\,u=\nabla_{\G}\,u+\nu\pa_\nu\,u,\quad \nabla\cdot v=\mbox{div}_\G\,v+\nu\cdot\pa_\nu\,v.
\en
In this case, the first integrand on the right hand side of (\ref{eq:17}) can be rewritten as
\ben
(h_j\cdot\nabla\ov{v})\cdot(\nu\cdot\sigma_j(u_j)) -(h_j\cdot\nu)(\sigma_j(u_j):\nabla\ov{v}) = \sum_{i=1}^3(\bb{A}_j(i,:))^\top\cdot(\nu\times\nabla\,v_i)
= \sum_{i=1}^3\nabla_\G\,v_i\cdot((\bb{A}_j(i,:))^\top\times\nu),
\enn
where the matrix $\bb{A}_j$ is given by (\ref{AJ}). Hence, by
 integration by part we find
\be\label{eq:18}
b_j(u,v,h) &=& A_{\Omega_j}(h\cdot\nabla u,v)+\rho_j\omega^2 \int_{\Gamma_j} (h_j\cdot\nu)(u_j\cdot\ov{v})ds -\int_{\Gamma_j} \mbox{div}_{\G_j} (\bb{A}_j\times\nu)\cdot\ov{v}ds
\en
for $j\geq 1$. Analogously,
\be\no
b_0(u,v,h) &=& A_{\Om_0}(h\cdot\nabla u,w)\\ \label{eq:19}
&&- \sum_{j=1}^{N_0}\left\{\rho_0\omega^2 \int_{\Gamma_j} (h_j\cdot\nu)[(u^{sc}+u^{in})\cdot\ov{w}]ds -\int_{\Gamma_j} \mbox{div}_{\G_j}(\bb{A}_0\times\nu)\cdot\ov{w}ds \right\},
\en
 From (\ref{eq:18}) and (\ref{eq:19}) we conclude the variational formulation (\ref{variational2}) still holds with $\widetilde{u}=w-h\cdot\nabla u$ in three dimensions. Moreover, we get $\widetilde{u}_0|_{\Gamma_R}=w|_{\Gamma_R}$ and the transmission conditions (\ref{derivative3})  due to the fact that
$\mbox{div}_{\G_j}\,u_j=\mbox{div}_{\G_j} (u^{sc}-u^{in})$.
This completes the proof.
\end{proof}

\subsection{Inversion algorithm in 2D}
In this subsection we design a descent algorithm for the inverse problem in two dimensions. Assume that $\Gamma_l$ ($l=1,2,\cdots, N_0$) is a star-shaped boundary that can be parameterized by $\gamma_l(\theta)$ as follows
\ben
\textcolor{rot}{\Gamma_l} =\{x\in\R^2:x=\gamma^{(l)}(\theta):=(a^{(l)}_1,a^{(l)}_2)^\top + r^{(l)}(\theta)(\cos\theta,\sin\theta)^\top,\theta\in[0,2\pi]\}
\enn
where the function $r^{(l)}$ is $2\pi$-periodic and twice continuously differentiable. Let the Fourier series expansion of $r^{(l)}$ be given by
\ben
r^{(l)}(\theta)=\alpha_0^{(l)}+\textcolor{rot}{\sum_{m=1}^\infty \left[\alpha^{(l)}_{2m-1}\cos(m\theta)+\alpha^{(l)}_{2m}\sin(m\theta)\right]}.
\enn
We approximate the unknown boundary $\textcolor{rot}{\Gamma_l}$ by the surface
\be\nonumber
\Gamma^{(l)}_M&=&\{x\in\R^2:x=\gamma^{(l)}(\theta):=(a^{(l)}_1,a^{(l)}_2)^\top + r^{(l)}_M(\theta)(\cos\theta,\sin\theta)^\top,\theta\in[0,2\pi]\},\\ \label{M}
r^{(l)}_M(\theta)&=&\alpha_0^{(l)}+\textcolor{rot}{\sum_{m=1}^M \left[\alpha^{(l)}_{2m-1}\cos(m\theta)+\alpha^{(l)}_{2m}\sin(m\theta)\right]},
\en
 in a finite dimensional space. The function $r^{(l)}_M$ is a truncated series of $r^{(l)}$.
For large $M$, the surface $\Gamma^{(l)}_M$ differs from $\textcolor{rot}{\Gamma_l}$ only in those  high frequency modes of $l\geq M$.
Evidently, there are totally $2M+3$ unknown parameters for $\Gamma^{(l)}_M$, which we denote by
 \ben \Lambda^{(l)}=(\Lambda^{(l)}_1,\cdots,\Lambda^{(l)}_{2M+3})^\top :=(a^{(l)}_1,a^{(l)}_2,\alpha^{(l)}_0,\alpha^{(l)}_1,\alpha^{(l)}_2, \cdots,\alpha^{(l)}_{2{M}-1},\alpha^{(l)}_{2{M}})^\top\in\C^{2{M}+3}.\enn
 Assume that the measurement points $\{z_i\}_{i=1}^{N_{mea}}$ are uniformly distributed on $\Gamma_R$,
 that is, $z_i=R(\cos\theta_i,\sin\theta_i)^\top$, $\theta_i=(i-1)2\pi/N_{mea}$. We use the notation $u(\cdot,d,\om)$ to denote the dependence of the total field on the incident direction $d$ and frequency $\om$. It is supposed  that the measured data are available over
 a finite number of frequencies $\omega_l\in [\omega_{min},\omega_{max}]$ ($l=1,2,\cdots,K$) and several incident directions $d_j$ $(j=1,\cdots,N_{inc})$.
 Hence, we have the data set of the total field
 \ben
 U_{mea}:=\{u(z_i,d_j,\om_l): i=1,\cdots,N_{mea},\; j=1,\cdots,N_{inc},\; l=1,\cdots,K\}.
 \enn
 Then we consider the following modified inverse problem:
\begin{description}
\item[\bf (IP'):] Determine the parameter vector $\Lambda^{(j)}$ of the boundary $\Gamma_j$, $j=1,\cdots,{N_0}$, from knowledge of the near-field data set $U_{mea}$.
\end{description}

The inverse problem can be formulated as the nonlinear operator equation
\be\label{operator}
\mathcal{J}(\Lambda^{(1)},\cdots,\Lambda^{(N_0)})=U_{mea},
\en where $\mathcal{J}$ is the solution operator for all incident directions $d_j$ and frequencies $\omega_l$. The data set can be rewritten as $U_{mea}=\cup_{i=1}^{N_{mea}} u_{mea}(z_i)$, where
\ben u_{mea}(z_i):=\{u(z_i,d_j,\omega_l): j=1,\cdots,N_{inc},\; l=1,\cdots,K\}
\enn is the data set at $z_i\in \Gamma_R$ over all  $d_j$ and $\omega_l$. Let $J_i$ be the solution operator mapping the boundary to $u_{mea}(z_i)$, i.e., $\mathcal{J}_i(\Lambda^{(1)},\cdots,\Lambda^{(N_0)})=u_{mea}(z_i)$.

 To solve the problem (\ref{operator}), we consider the objective function
\ben
F(\Lambda^{(1)},\cdots,\Lambda^{({N_0})}) = \frac{1}{2} \|\mathcal{J}(\Lambda^{(1)},\cdots,\Lambda^{({N_0})})-U_{mea}\|_{l^2}.
\enn
Then the inverse problem ({\bf IP'}) can be formulated as the minimization problem
\ben
\min_{\Lambda^{(1)},\cdots,\Lambda^{({N_0})}\in\C^{2M+3}} F(\Lambda^{(1)},\cdots,\Lambda^{({N_0})}).
\enn
To apply the descent method, it is necessary to compute the gradient of the objective function.
A direct calculation yields that
\ben
\frac{\partial F(\Lambda^{(1)},\cdots,\Lambda^{({N_0})})}{\partial \Lambda^{(l)}_n}=\mbox{Re}\left\{\sum_{i=1}^{N_{mea}}\frac{\partial \mathcal{J}_{i}(\Lambda^{(1)},\cdots,\Lambda^{({N_0})})}{\partial \Lambda^{(l)}_n}\cdot \left[\mathcal{J}_{i}(\Lambda^{(1)},\cdots,\Lambda^{({N_0})})-u_{mea}(z_i)\right]  \right\}.
\enn
Set
\ben
\nabla_{\Lambda^{(l)}}F:=\left(\frac{\partial F(\textcolor{rot}{\Lambda^{(1)}},\cdots,\Lambda^{({N_0})})}{\partial \Lambda^{(l)}_1},\cdots, \frac{\partial F(\Lambda^{(1)},\cdots,\Lambda^{({N_0})})}{\partial \Lambda^{(l)}_{2M+3}}\right)^\top,\quad l=1,2,\cdots, N_0.
\enn
The calculation of $\nabla_{\Lambda^{(l)}}F$ is based on Theorem \ref{theorem4.1} below, which is a consequence of Theorem \ref{theorem3.1}.
\begin{theorem}
\label{theorem4.1}
Let $u$ be the unique solution of the variational problem (\ref{variational2}) with fixed incident direction and frequency. Then the operator $\mathcal{J}_i$ is differentiable in $\Lambda^{(l)}_n$ and its derivatives are given by
\ben
\frac{\partial \mathcal{J}_{i}(\Lambda^{(1)},\cdots,\Lambda^{({N_0})})}{\partial \Lambda^{(l)}_n} =\widetilde{u}_0(z_i),\; l=1,\cdots,{N_0},\;n=1,\cdots,2M+3,\;i=1,\cdots,N_{mea},
\enn
where $\widetilde{u}_0$, together with $\widetilde{u}_j$ ($j=1,\cdots,N_0$),  is the unique weak solution of the boundary value problem:
\ben
\label{Dderivative1}
\nabla\cdot (\mathcal{C}_j:\nabla \widetilde{u}_j )+\rho_j\omega^2 \widetilde{u}_j = 0&&\mbox{in}\quad\Omega_j,\; j=0, 1,\cdots,{N_0},\\
\label{Dderivative3}
\widetilde{u}_j-\widetilde{u}_0-f_j = 0 &&\mbox{on}\quad\Gamma_j,\;j=1,\cdots,{N_0},\\
\label{Dderivative4}
\mathcal{N}_{\mathcal{C}}^-\widetilde{u}_j -\mathcal{N}_\mathcal{C}^+\widetilde{u}_0-g_j =0 &&\mbox{on}\quad\Gamma_j,\;j=1,\cdots,{N_0},\\
\label{Dderivative5}
T \widetilde{u}_0 -\mathcal{T}\widetilde{u}_0=0&&\mbox{on}\quad\Gamma_R.
\enn
Here, $f_j=g_j=0$ for $j=1,\cdots,N_0$, $j\ne l$, and
\ben
\label{Ddataf}
f_l &=& -(h_l\cdot\nu)\left[{\partial_\nu^- u_l}-{\partial_\nu^+ (u^{sc}+u^{in})}\right],\\
\label{Ddatag}
g_l &=& \omega^2(h_l\cdot\nu)\left[\rho_l {u_l}^--\rho_0 (u^{sc}+u^{in})^+\right]\nonumber\\
&-&\partial_\tau\left[\left((\sigma_l(u_l))^--(\sigma_0(u^{sc}+u^{in}))^+\right) (h_{l,2},-h_{l,1})^\top\right],
\enn
 where $u_l:=u|_{\Omega_l}$ and the functions $h_{l,1}$, $h_{l,2}$ are defined in the following way relying on $n$:
\ben
h_{l,1}(\theta)=\begin{cases}
1, & n=1, \cr
0, & n=2, \cr
\cos\theta, & n=3, \cr
\cos((n-2)\theta/2)\cos\theta, & n=4,6,8,\cdots,2M+2, \cr
\sin((n-3)\theta/2)\cos\theta, & n=5,7,9,\cdots,2M+3,
\end{cases}
\enn
\ben
h_{l,2}(\theta)=\begin{cases}
0, & n=1, \cr
1, & n=2, \cr
\sin\theta, & n=3, \cr
\cos((n-2)\theta/2)\sin\theta, & n=4,6,8,\cdots,2M+2, \cr
\sin((n-3)\theta/2)\sin\theta, & n=5,7,9,\cdots,2M+3.
\end{cases}
\enn
\end{theorem}

We now propose an algorithm based on the descent method to reconstruct the coefficient vectors $\Lambda^{(l)}$, $l=1,\cdots,{N_0}$. We assume the number $N_0$ of the disconnected components is known in advance. For notational convenience we denote by $\Lambda^{(l,i,j,m)}$ the solution of the inverse problem at the $i$-th iteration step reconstructed from the data set at the frequency $\om_m$ with the incident direction $d_j$.
Our approach consists of the following steps:
\begin{description}
\item[Step 1.] Collect the near-field data over  all frequencies $\om_m$, $m=1,\cdots,K$ and all incident directions $d_j$, $j=1,\cdots,N_{inc}$.
\item[Step 2.] Set initial approximations $\Lambda^{(l,0,0,0)}$ for every $l=1,\cdots,{N_0}$.
\item[Step 3.] For all $l=1,\cdots,N_0$, update the coefficient vector by the iterative formula
    \ben
    \Lambda^{(l,i+1,j,m)}=\Lambda^{(l,i,j,m)}-\epsilon \nabla_{\Lambda^{(l,i,j,m)}}F,\quad i=0,\cdots,L-1,
    \enn
    where $\epsilon$ and $L > 0$ are the step size and total number of iterations, respectively.
\item[Step 4.] For all $l=1,\cdots,{N_0}$, set $\Lambda^{(l,0,j+1,m)}=\Lambda^{(l,L,j,m)}$ and repeat Step 3 until the last incident directions $d_{N_{inc}}$ is reached.
\item[Step 5.] For all $l=1,\cdots,{N_0}$, set $\Lambda^{(l,0,0,m+1)}=\Lambda^{(l,L,N_{inc},m)}$. Repeat Step 3 from the smallest frequency $\omega_1$ and end up with the highest frequency $\omega_K$.
\end{description}

\section{Numerical examples}\label{sec:Numerics}
In this section, we present several numerical examples in 2D to verify the efficiency and validity of \textcolor{rot}{the finite element method solving direct scattering problems and the reconstruction scheme for inverse scattering problems}.

\subsection{Numerical solutions to direct scattering problems}
Firstly, we present an analytic solution to the elastic wave equation in a homogeneous anisotropic medium; see \cite[Chapter 1.7.1]{PS01} for the details. Such a solution will be used to verify the accuracy of our numerical scheme. For simplicity we assume that $\Omega$ consists of one component only, i.e., $N_0=1$.

 In 2D, the symmetry of the stiffness tensor $\mathcal{C}=\{C_{ijkl}\}_{i,j,k,l=1}^2$ leads to at most 6 different elements of stiffness. Using the Voigt notation for tensor indices, i.e,
\ben
ij &=& 11 \quad 22 \quad 12,21\\
\Downarrow &\quad& \,\Downarrow\quad\;\, \Downarrow\quad\quad \Downarrow\\
\alpha &=& \,1 \quad\;\; 2 \quad\quad\, 3
\enn
one can rewrite the stiffness tensor as
\ben
C_{ijkl} \Rightarrow C_{\alpha\beta}=\begin{bmatrix}
C_{11} & C_{12} & C_{13} \\
C_{12} & C_{22} & C_{23} \\
C_{13} & C_{23} & C_{33} \\
\end{bmatrix}.
\enn
In particular, we have
\ben
C_{\alpha\beta}=\begin{bmatrix}
\lambda+2\mu & \lambda & 0 \\
\lambda & \lambda+2\mu & 0 \\
0 & 0 & \mu \\
\end{bmatrix},
\enn
if the elastic medium is homogeneous isotropic with Lam\'e constants $\lambda$ and $\mu$.

In a homogeneous anisotropic medium, we consider the propagation of a plane wave which is perpendicular to a fixed unit vector $d=(d_1,d_2)^\top\in\textcolor{rot}{\mathbb{S}^1}$. The plane wave takes the form
\be
\label{generalsol}
u=p\,e^{i\frac{\omega}{v_p}x\cdot d},
\en
where $p=(p_1,p_2)^\top$ and $v_p$ are the polarization vector and phase velocity to be determined, respectively. Inserting the solution (\ref{generalsol}) into the elastic equation (\ref{Navier-C})  gives
\ben
A_C\,p=\rho\, v_p^2\, p,
\enn
where
\be\nonumber
A_C &=& \left\{\sum_{k,l=1}^2 C_{iklj}d_kd_l\right\}_{i,j=1}^2\\ \label{AC}
 &=& \begin{bmatrix}
C_{11} & C_{13} \\
C_{13} & C_{33} \\
\end{bmatrix}d_1^2+\begin{bmatrix}
2C_{13} & C_{12}+C_{33} \\
C_{12}+C_{33} & 2C_{23} \\
\end{bmatrix}d_1d_2+\begin{bmatrix}
C_{33} & C_{23} \\
C_{23} & C_{22} \\
\end{bmatrix}d_2^2.
\en
It follows from the uniform Legendre ellipticity condition of $\mathcal{C}$ that the matrix $A_C$ is positive definite. Thus, the eigenvectors of $A_C$ give the vector $p$ with the corresponding eigenvalue $\rho v_p^2$.

In order to check whether our code provides the true solution,  we consider the elastic transmission problem: {\it Given $f\in \textcolor{rot}{(H^{1/2}(\Gamma))^2}$ and $g\in \textcolor{rot}{(H^{-1/2}(\Gamma))^2}$, find $u\in \textcolor{rot}{(H^1(\Omega))^2}$ and $\textcolor{rot}{u^{sc}\in(H^1_{loc}(\Omega^c))^2}$ such that}
\be
\label{num1-1}
\nabla\cdot\left(\mathcal{C}:\nabla u\right)+\rho\omega^2 u &=& 0 \quad\mbox{in}\quad\Omega,\\
\label{num1-2}
\Delta^* u^{sc}+\rho_0\omega^2 u^{sc} &=& 0 \quad\mbox{in}\quad\Omega^c,\\
\label{num1-3}
u-u^{sc}&=& f \quad\mbox{on}\quad\pa\Omega,\\
\label{num1-4}
\mathcal{N}_\mathcal{C}^-u-T_{\lambda,\mu}u^{sc} &=& g \quad\mbox{on}\quad\pa\Omega,
\en
and the scattered field $u^{sc}$ satisfies the Kupradze radiation condition. If $\Omega$ is specified as a homogeneous isotropic medium characterized by the density $\rho_1>0$ and the Lam\'e constants $\lambda_1$ and $\mu_1$ \textcolor{rot}{are such that} $\mu_1>0$ and $\lambda_1+\mu_1>0$, then the problem (\ref{num1-1})-(\ref{num1-4}) is reduced to
\be
\label{num2-1}
\Delta^*_1 u+\rho_1\omega^2 u &=& 0 \quad\mbox{in}\quad\Omega,\\
\label{num2-2}
\Delta^* u^{sc}+\rho_0\omega^2 u^{sc} &=& 0 \quad\mbox{in}\quad\Omega^c,\\
\label{num2-3}
u-u^{sc}&=& f \quad\mbox{on}\quad\pa\Omega,\\
\label{num2-4}
T_{\lambda_1,\mu_1}u-T_{\lambda,\mu}u^{sc} &=& g \quad\mbox{on}\quad\pa\Omega,
\en
where $\Delta_1^*:=\mu_1\Delta +(\lambda_1+\mu_1)\mbox{grad}\,\mbox{div}$.

We define the far-field pattern of the total displacement as
\ben
u^\infty(\widehat{x})=u_p^\infty(\widehat{x})\,\widehat{x} +u_s^\infty(\widehat{x})\,\widehat{x}^\perp,
\enn
where $u_p^\infty(\widehat{x})=u^\infty(\widehat{x}) \cdot\widehat{x}$, $u_s^\infty(\widehat{x})=u^\infty(\widehat{x})\cdot\widehat{x}^\perp$ are two scalar functions given by the asymptotic behavior
\ben
u^{sc}=\frac{\exp(ik_px+i\pi/4)}{\sqrt{8\pi k_p|x|}} u_p^\infty(\widehat{x})\,\widehat{x} +\frac{\exp(ik_sx+i\pi/4)}{\sqrt{8\pi k_s|x|}} u_s^\infty(\widehat{x})\,\widehat{x}^\perp +O(|x|^{-3/2}).
\enn
We decompose the scattered field into
\ben
u^{sc}=\mbox{grad}\,\Psi_p+\overrightarrow{\mbox{curl}}\,\Psi_s,
\enn
where
\ben
\Psi_p=\sum_{n\in\Z}\Psi_p^n\, H_n^{(1)}(k_p|x|)e^{in\theta_x},\quad
\Psi_s=\sum_{n\in\Z}\Psi_s^n\, H_n^{(1)}(k_s|x|)e^{in\theta_x},\qquad \Psi_p^n,\;\Psi_s^n\in \C.
\enn
Then it follows from the asymptotic behavior of Hankel functions that
\ben
&&u_p^\infty(\widehat{x})= u_p^\infty(\theta)=4k_p\sum_{n\in\Z}\Psi_p^n\,e^{in(\theta-\pi/2)},\\
&&u_s^\infty(\widehat{x})= u_s^\infty(\theta)=-4k_s\sum_{n\in\Z}\Psi_s^n\,e^{in(\theta-\pi/2)}.
\enn

\textcolor{rot}{In numerical computations, the computational domains $\Omega$ and $\Omega_0$ are discretized by uniform triangle elements and we employ piecewise linear basis functions to construct the finite element space of $(H^1(\Omega))^2$ and $(H^1(\Omega_0))^2$.}

{\bf Example 1.} In the first example, $\Omega$ is specified as a homogeneous isotropic medium and we consider the problem (\ref{num2-1})-(\ref{num2-4}).
Let $f$ and $g$ be such that the analytic solution of the above boundary value problem is given by
\ben
u(x) = \nabla J_0(k_{p,1}|x|),\quad x\in\Omega,\quad\quad
u^{sc}(x) = \nabla H_0^{(1)}(k_p|x|),\quad x\in\Omega^c,
\enn
where $k_{p,1}=\omega\sqrt{\rho_1/(\lambda_1+2\mu_1)}$. We choose $\lambda_1=2$, $\mu_1=3$, $\rho_1=3$, $\lambda=1$, $\mu=2$, $\textcolor{rot}{\rho_0=1}$ and the boundary $\pa\Omega$ is selected to be a circle
\ben
\pa\Omega=\{x\in\R^2: |x|=1\},
\enn
or a rounded-triangle-shaped curve
\ben
\pa\Omega=\{x\in\R^2: x=(2+0.5\cos3t)(\cos t,\sin t)\textcolor{rot}{^\top}, t\in[0,2\pi)\}.
\enn
 Denote $U=(u,u^{sc})$ and $U_h=(u_h,u_h^{sc})$ the exact and numerical solutions, respectively. The numerical errors (see Tables \ref{Table1} and \ref{Table2})
\be\label{eq:20}
E_0=\|U-U_h\|_{\textcolor{rot}{(L^2(\Omega))^2\times (L^2(\Om_0))^2}},\quad E_1=\|U-U_h\|_{\textcolor{rot}{(H^1(\Omega))^2\times (H^1(\Om_0))^2}},
\en
\textcolor{rot}{indicate the convergence order}
\be
\label{order}
E_0=O(h^2),\quad E_1=O(h),
\en
where $h$ denotes the \textcolor{rot}{finite element} mesh size for discretizing our variational formulation.

\begin{table}[ht]
\centering
\begin{tabular}{p{10pt} p{30pt} p{50pt} p{30pt} p{50pt} p{30pt}}
\multicolumn{6}{p{260pt}}{}  \\
\hline
$\omega$ & $h$ & $E_0$ & Order &  $E_1$ & Order \\
\hline
  &$h_0$    & 1.55E-2 & --     & 2.29E-1 & --    \\
1 &$h_0/2$  & 3.97E-3 & 1.97   & 1.07E-1 & 1.10  \\
  &$h_0/4$  & 1.00E-3 & 1.99   & 5.22E-2 & 1.04  \\
\hline
  &$h_0$    & 1.52E-1 & --     & 6.66E-1 & --    \\
3 &$h_0/2$  & 3.81E-2 & 2.00   & 2.95E-1 & 1.17  \\
  &$h_0/4$  & 9.62E-3 & 1.99   & 1.43E-1 & 1.04  \\
\hline
\end{tabular}\caption{Numerical errors for Example 1 where $\Gamma$ is a circle, $h_0=0.4304$ and $R=2$.}
\label{Table1}
\end{table}

\begin{table}[ht]
\centering
\begin{tabular}{p{10pt} p{30pt} p{50pt} p{30pt} p{50pt} p{30pt}}
\multicolumn{6}{p{260pt}}{}  \\
\hline
$\omega$ & $h$ & $E_0$ & Order &  $E_1$ & Order \\
\hline
  &$h_0$    & 3.03E-2 & --     & 1.56E-1 & --    \\
1 &$h_0/2$  & 6.83E-3 & 2.15   & 7.16E-2 & 1.12  \\
  &$h_0/4$  & 1.79E-3 & 1.93   & 3.56E-2 & 1.01  \\
\hline
  &$h_0$    & 5.22E-1 & --     & 1.94E0  & --    \\
3 &$h_0/2$  & 1.32E-1 & 1.98   & 7.81E-1 & 1.31  \\
  &$h_0/4$  & 3.59E-2 & 1.88   & 3.75E-1 & 1.06  \\
\hline
\end{tabular}\caption{Numerical errors for Example 1 where $\Gamma$ is a rounded-triangle-shaped curve, $h_0=1.1474$ and $R=5$.}
\label{Table2}
\end{table}

{\bf Example 2.} In this example, $\Omega$ is supposed to be a homogeneous anisotropic medium characterized by the density $\textcolor{rot}{\rho}>0$ and the stiffness tensor
\ben
C_{\alpha\beta}=\begin{bmatrix}
C_{11} & C_{12} & C_{13} \\
C_{12} & C_{22} & C_{23} \\
C_{13} & C_{23} & C_{33} \\
\end{bmatrix}.
\enn
Consider the problem (\ref{num1-1})-(\ref{num1-4}) and let $f$ and $g$ be such that the analytic solution is given by
\ben
u(x) = pe^{i\frac{\omega}{v_p}x\cdot d},\quad x\in\Omega,\qquad
u^{sc}(x) = \nabla H_0^{(1)}(k_p|x|),\quad x\in\Omega^c,
\enn
where $d=(\sqrt{2}/2,\sqrt{2}/2)$ and $\textcolor{rot}{\rho} v_p^2$ is the first eigenvalue of the matrix $A_C$ (see (\ref{AC}) ). We choose
\ben
&C_{11}=10.5, C_{22}=13, C_{33}=4.75, C_{12}=3.25, C_{13}=-0.65, C_{23}=-1.52, \\
&\textcolor{rot}{\rho}=3, \lambda=1, \mu=2, \textcolor{rot}{\rho_0=1}.
\enn
The boundary $\pa\Omega$ is selected to be a circle or a rounded-triangle-shaped curve given in Example 1.
In Tables \ref{Table3} and \ref{Table4} we illustrate the
the numerical errors of $E_0$ and $E_1$ (see (\ref{eq:20}))
 which also \textcolor{rot}{indicate the convergence order (\ref{order})}. We plot the
the numerical solutions in Figures \ref{E2-1} and \ref{E2-2} from which it can be seen that they are in a good agreement with the exact ones. To compare the errors  for far-field patterns,  we observe that the exact far-field pattern takes the explicit form
$
u^\infty(\widehat{x})=4k_p\widehat{x}$.
From Figures \ref{E2-3} and \ref{E2-4} it can be seen that the numerical far-field patterns provide good approximations to the exact ones.

\begin{table}[htb]
\centering
\begin{tabular}{p{10pt} p{30pt} p{50pt} p{30pt} p{50pt} p{30pt}}
\multicolumn{6}{p{260pt}}{}  \\
\hline
$\omega$ & $h$ & $E_0$ & Order &  $E_1$ & Order \\
\hline
  &$h_0$    & 1.95E-2 & --     & 2.46E-1 & --    \\
1 &$h_0/2$  & 5.11E-3 & 1.93   & 1.15E-1 & 1.10  \\
  &$h_0/4$  & 1.32E-3 & 1.95   & 5.60E-2 & 1.04  \\
\hline
  &$h_0$    & 2.86E-1 & --     & 1.27E0  & --    \\
3 &$h_0/2$  & 8.62E-2 & 1.73   & 5.15E-1 & 1.30  \\
  &$h_0/4$  & 2.32E-2 & 1.89   & 2.28E-1 & 1.18  \\
\hline
\end{tabular}\caption{Numerical errors for Example 2 where $\Gamma$ is a circle, $h_0=0.4304$ and $R=2$.}
\label{Table3}
\end{table}

\begin{table}[htb]
\centering
\begin{tabular}{p{10pt} p{30pt} p{50pt} p{30pt} p{50pt} p{30pt}}
\multicolumn{6}{p{260pt}}{}  \\
\hline
$\omega$ & $h$ & $E_0$ & Order &  $E_1$ & Order \\
\hline
  &$h_0$    & 9.08E-2 & --     & 3.12E-1 & --    \\
1 &$h_0/2$  & 2.26E-2 & 1.97   & 1.30E-1 & 1.10  \\
  &$h_0/4$  & 5.87E-3 & 1.99   & 6.30E-2 & 1.04  \\
\hline
  &$h_0$    & 1.82E0  & --     & 5.78E0  & --    \\
3 &$h_0/2$  & 5.73E-1 & 2.00   & 2.04E0  & 1.17  \\
  &$h_0/4$  & 1.59E-1 & 1.99   & 7.35E-1 & 1.04  \\
\hline
\end{tabular}\caption{Numerical errors for Example 2 where $\Gamma$ is a rounded-triangle-shaped curve,  $h_0=1.1474$ and $R=3$.}
\label{Table4}
\end{table}

\begin{figure}[htb]
\centering
\begin{tabular}{cccc}
\includegraphics[scale=0.15]{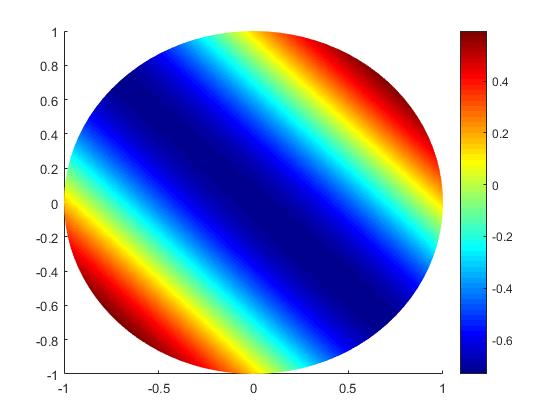}&
\includegraphics[scale=0.15]{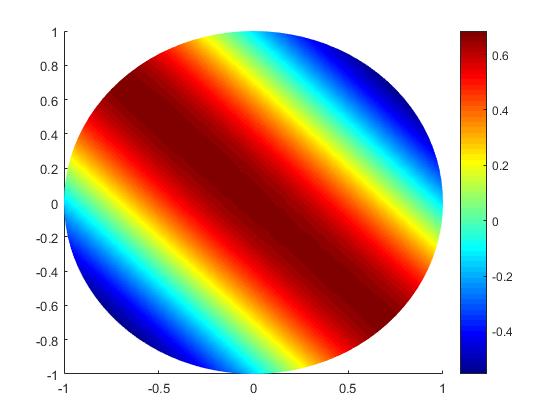}&
\includegraphics[scale=0.15]{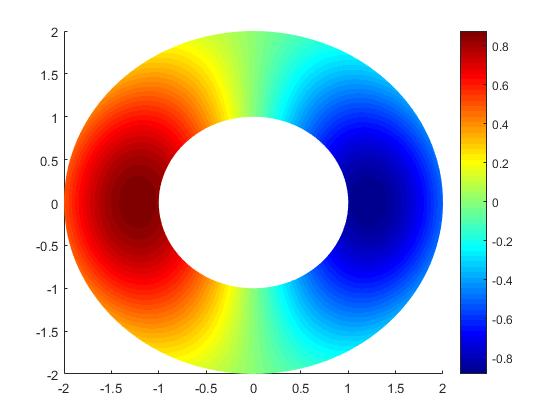}&
\includegraphics[scale=0.15]{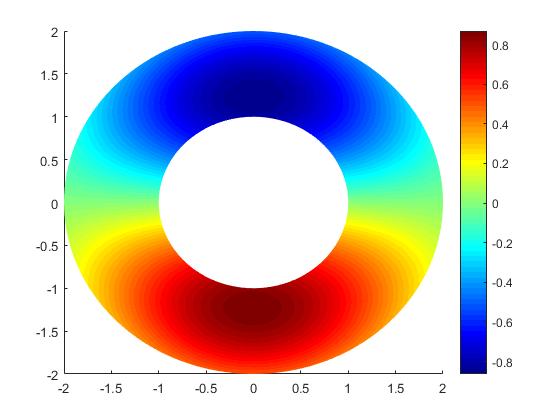}\\
(a) $\mbox{Re}\,u_{h,1}$ & (b) $\mbox{Re}\,u_{h,2}$ & (c) $\mbox{Re}\,u^{sc}_{h,1}$ & (d) $\mbox{Re}\,u^{sc}_{h,2}$ \\
\includegraphics[scale=0.15]{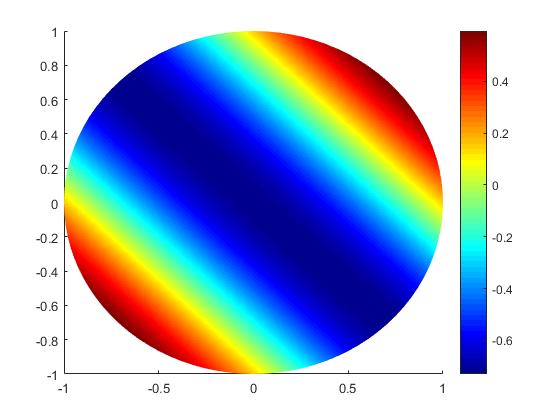}&
\includegraphics[scale=0.15]{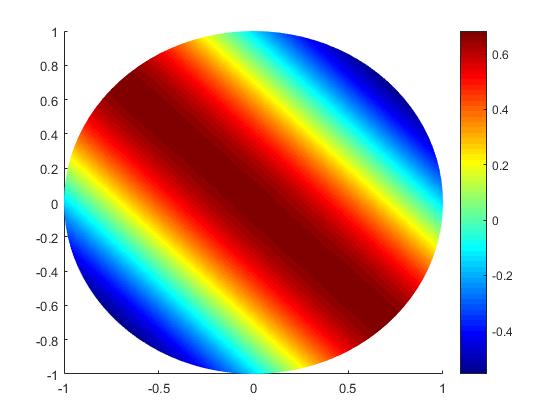}&
\includegraphics[scale=0.15]{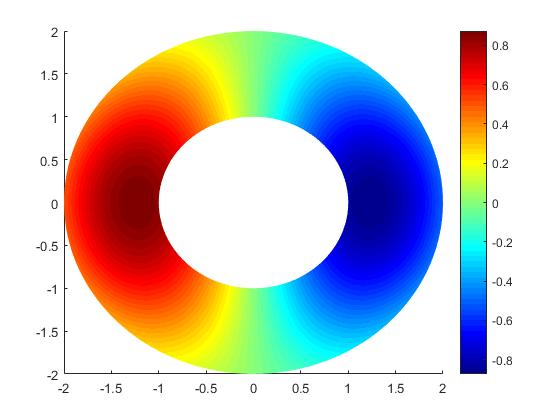}&
\includegraphics[scale=0.15]{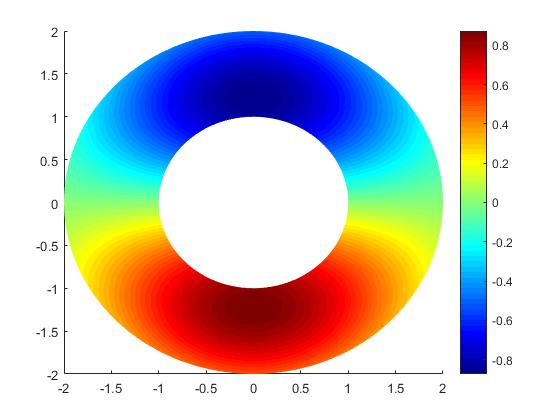}\\
(e) $\mbox{Re}\,u_{1}$ & (f) $\mbox{Re}\,u_{2}$ & (g) $\mbox{Re}\,u^{sc}_{1}$ & (h) $\mbox{Re}\,u^{sc}_{2}$
\end{tabular}
\caption{Real parts of the numerical solutions $u_h=(u_{h,1},u_{h,2})^\top$, $u^{sc}_h=(u^{sc}_{h,1},u^{sc}_{h,2})^\top$ and exact solutions $u=(u_{1},u_{2})^\top$, $u^{sc}=(u^{sc}_{1},u^{sc}_{2})^\top$ for Example 2. We set $\omega=3$,  $h=0.1076$ and $R=2$.}
\label{E2-1}
\end{figure}

\begin{figure}[htb]
\centering
\begin{tabular}{cccc}
\includegraphics[scale=0.15]{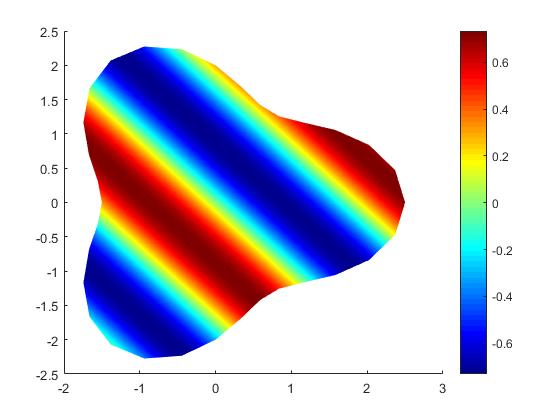}&
\includegraphics[scale=0.15]{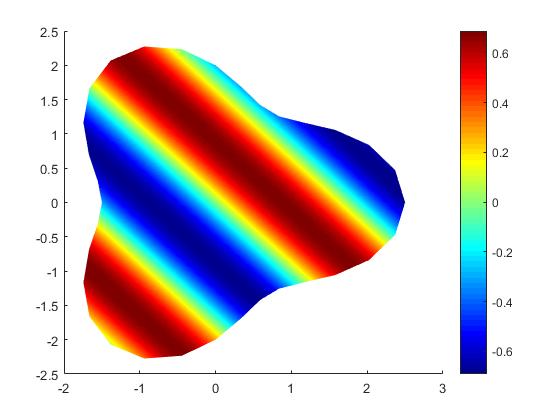}&
\includegraphics[scale=0.15]{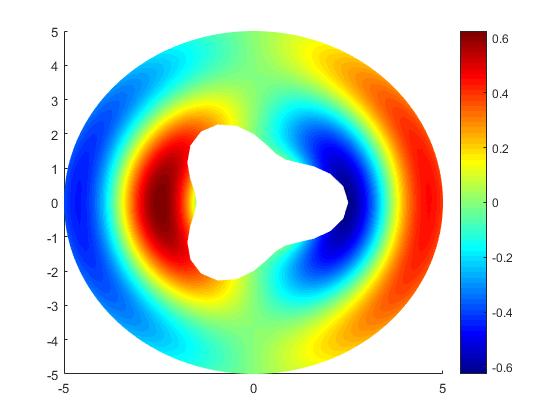}&
\includegraphics[scale=0.15]{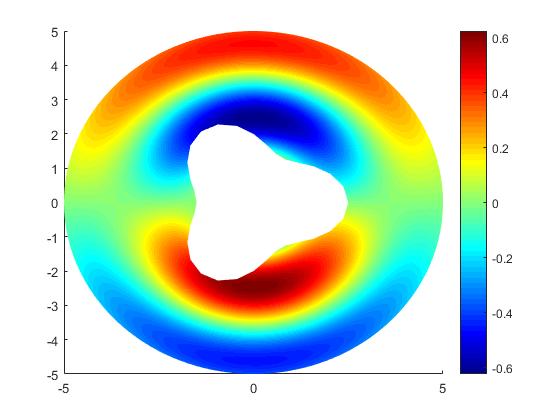}\\
(a) $\mbox{Im}\,u_{h,1}$ & (b) $\mbox{Im}\,u_{h,2}$ & (c) $\mbox{Im}\,u^{sc}_{h,1}$ & (d) $\mbox{Im}\,u^{sc}_{h,2}$ \\
\includegraphics[scale=0.15]{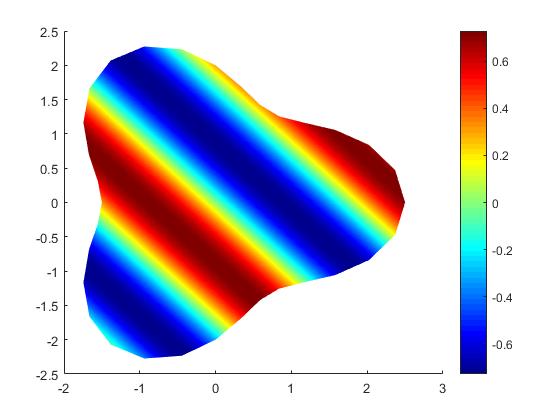}&
\includegraphics[scale=0.15]{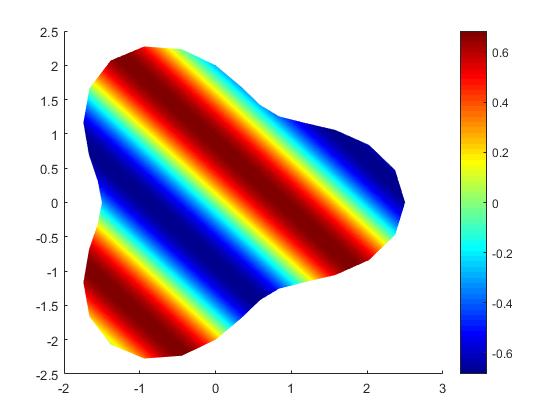}&
\includegraphics[scale=0.15]{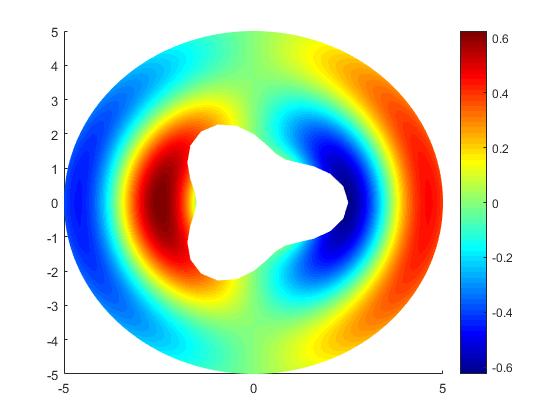}&
\includegraphics[scale=0.15]{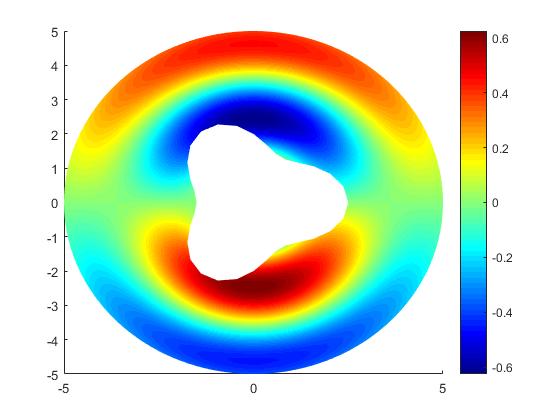}\\
(e) $\mbox{Im}\,u_{1}$ & (f) $\mbox{Im}\,u_{2}$ & (g) $\mbox{Im}\,u^{sc}_{1}$ & (h) $\mbox{Im}\,u^{sc}_{2}$
\end{tabular}
\caption{Imaginary parts of the numerical solutions $u_h=(u_{h,1},u_{h,2})^\top$, $u^{sc}_h=(u^{sc}_{h,1},u^{sc}_{h,2})^\top$ and exact solutions $u=(u_{1},u_{2})^\top$, $u^{sc}=(u^{sc}_{1},u^{sc}_{2})^\top$ for Example 2. We set $\omega=3$,  $h=0.2869$ and $R=3$.}
\label{E2-2}
\end{figure}

\begin{figure}[htb]
\centering
\begin{tabular}{cc}
\includegraphics[scale=0.4]{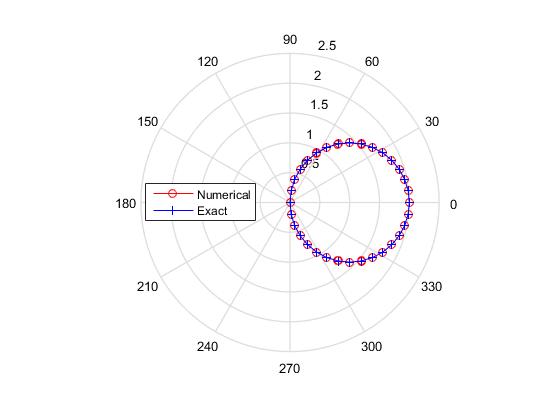}&
\includegraphics[scale=0.4]{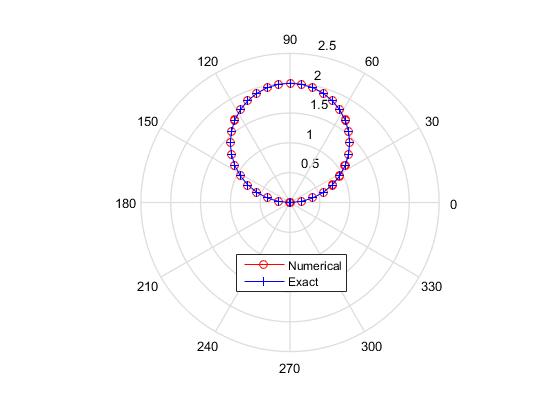}\\
(a) $\mbox{Re}\,u_1^\infty$ & (b) $\mbox{Re}\,u_2^\infty$ 
\end{tabular}
\caption{Exact and numerical far-field pattern $u^\infty=(u_1^\infty,u_2^\infty)^\top$ for Example 2 when $\Gamma$ is a circle.}
\label{E2-3}
\end{figure}

\begin{figure}[htb]
\centering
\begin{tabular}{cc}
\includegraphics[scale=0.4]{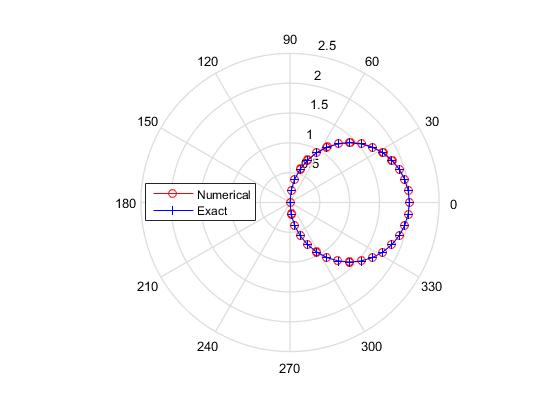}&
\includegraphics[scale=0.4]{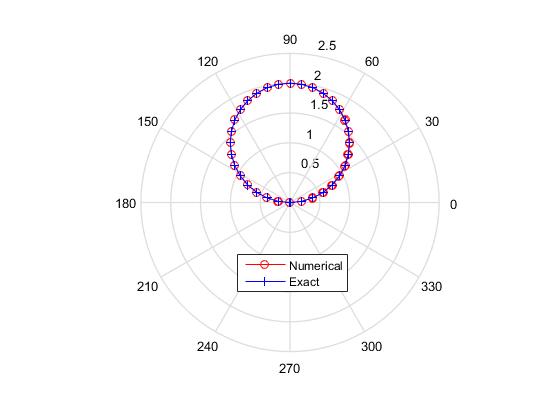}\\
(a) $\mbox{Re}\,u_1^\infty$ & (b) $\mbox{Re}\,u_2^\infty$ 
\end{tabular}
\caption{Exact and numerical far-field pattern $u^\infty=(u_1^\infty,u_2^\infty)^\top$ for Example 2 when $\Gamma$ is a rounded-triangle-shaped curve.}
\label{E2-4}
\end{figure}
\subsection{Numerical solutions to inverse scattering problems}
We consider the reconstruction of multiple anisotropic elastic bodies in 2D  using the inversion algorithm described in Section \ref{sec:inverse}. Set $\rho_0=2000$ Kg/m$^3$, $c_p=\sqrt{(\lambda+2\mu)/\rho_0}=3000$ m/s, $c_s=\sqrt{\mu/\rho_0}=1800$ m/s and $R=5$ m. The number of measurement points and iterations are taken as $N_{mea}=64$, $L=10$, respectively.
For each frequency, we set the step size as $\epsilon=0.005/k_p$. The boundary of the unknown anisotropic obstacles  together with the initial guess are illustrated in Figure \ref{obstacle}, in which Obstacle 1 is kite-shaped and Obstacle 2 is an ellipse. The density of the anisotropic medium is selected as $\rho=2400$ Kg/m$^3$. We choose the stiffness tensor as
\ben
C_{klmn}\Rightarrow C_{\alpha\beta}=\begin{bmatrix}
6 & 8 & 2 \\
8 & 21 & 10 \\
2 & 10 & 30 \\
\end{bmatrix}\times 10^{10}\;\mbox{Pa}.
\enn
To examine the reconstruction results, we compute the residual error $\mbox{RError}_{i+1}$, $i=0,\cdots,K$ of the total field 
where
\ben
\mbox{RError}_{i+1}=\frac{\left\|\mathcal{J}_N(\Lambda^{(1,L,N_{inc},i)},\cdots, \Lambda^{(N,L,N_{inc},i)})-U_{mea}\right\|_{l^2}}{\left\|U_{mea}\right\|_{l^2}}.
\enn

In the first experiment, we use four incident plane waves (i.e., $N_{inc}=4$) incited at two frequencies $\omega_1=5\,\mbox{kHz}$ and $\omega_2=6\,\mbox{kHz}$ (i.e. $K=2$). The reconstruction results at each frequency are shown in Figure \ref{E3-1}. For different choice of $M$ (see (\ref{M})), the residual errors listed in Table \ref{Table5} indicates that the residual error decreases as frequency increases. Note that the errors corresponding to $M=10$ and $M=20$ are almost the same, because the underlying scatterers possess smooth boundaries.

In the second experiment, we use the data generated by one fixed direction $d=(-\sqrt{2}/2,\sqrt{2}/2)^\top$ (i.e., $N_{inc}=1$) and by three distinct frequencies $\omega_1=5\,\mbox{kHz}$, $\omega_2=6\,\mbox{kHz}$ and $\omega_3=7\,\mbox{kHz}$ (i.e., $K=3$). In this case the number of iterations at each frequency is set as $L=20$. The parameter $M$ for truncating the Fourier series is taken as
$M=20$. The reconstruction results shown in Figure \ref{E3-2} are very satisfactory.

\begin{table}[htb]
\centering
\begin{tabular}{p{20pt} p{50pt} p{50pt} p{50pt}}
\multicolumn{4}{p{200pt}}{}  \\
\hline
$M$ & $\mbox{RError}_1$ & $\mbox{RError}_2$ & $\mbox{RError}_3$ \\
\hline
10  & 0.3042   & 0.0656  & 0.0521  \\
\hline
20  & 0.3042   & 0.0738  & 0.0528  \\
\hline
\end{tabular}\caption{Change of residual reconstruction errors with respect to frequencies.}
\label{Table5}
\end{table}

\begin{figure}[htb]
\centering
\includegraphics[scale=0.5]{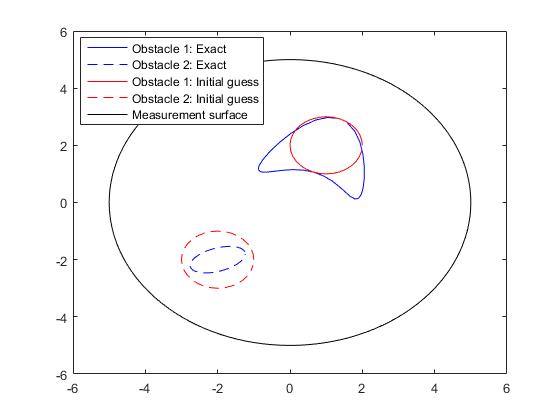}\\
\caption{The obstacles to be reconstructed and initial guess.}
\label{obstacle}
\end{figure}

\begin{figure}[htb]
\centering
\begin{tabular}{cc}
\includegraphics[scale=0.35]{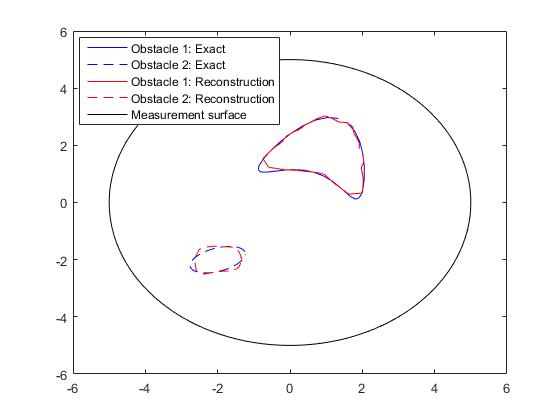} &
\includegraphics[scale=0.35]{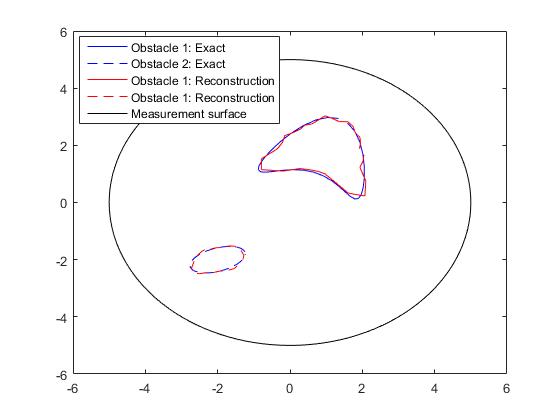} \\
(a) $\omega=5\mbox{kHz}$ & (b) $\omega=6\mbox{kHz}$
\end{tabular}
\caption{Reconstruction results from four incident directions at distinct frequencies. We set $M=10$.}
\label{E3-1}
\end{figure}

\begin{figure}[htb]
\centering
\includegraphics[scale=0.5]{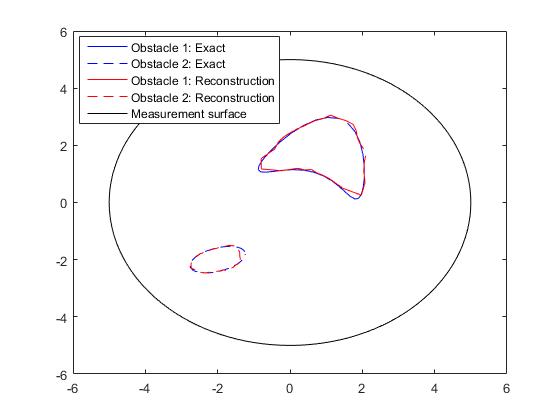} \\
\caption{Reconstruction result from the data of one incident direction and three frequencies.}
\label{E3-2}
\end{figure}

\section*{Acknowledgement}
The work of G. Bao is supported in part by a Key Project of the Major Research Plan of NSFC (No. 91130004), a NSFC A3 Project (No.11421110002), NSFC Tianyuan Projects (No. 11426235; No. 11526211, a NSFC Innovative Group Fun
(No.11621101), and a special research grant from Zhejiang University. The work of T. Yin is partially supported by the NSFC Grant (No. 11371385; No.  11501063).

\end{document}